\documentclass{amsart}
\numberwithin{equation}{section}
\usepackage[english]{babel}
\usepackage[T1]{fontenc}
\usepackage[latin1]{inputenc}
\usepackage{indentfirst}
\usepackage{enumitem}
\usepackage{amsmath,amssymb, amsbsy}

\usepackage{amsfonts}
\usepackage{hyperref}
\usepackage{latexsym}
\usepackage{amsthm}
\usepackage[dvips]{graphicx}
\usepackage{xcolor}
\usepackage{tikz}
\usepackage{pgfplots}
\usetikzlibrary{intersections}
\usepackage{tikz-3dplot}
\usepackage[outline]{contour}
\DeclareGraphicsExtensions{.pdf,.png,.jpg,.eps}

\usepackage{color}
\usepackage[latin1]{inputenc}
\usepackage[active]{srcltx}
\definecolor{Arancio}{cmyk}{0,0.61,0.87,0}

\definecolor{blus}{RGB}{0,102,204}

\newcommand{\brd}[1]{\mathbb{#1}}
\newcommand{\R}{\brd{R}}
\newcommand{\C}{\brd{C}}
\newcommand{\N}{\brd{N}}

\newcommand{\loc}{{{\tiny{\mbox{loc}}}}}

\newcommand{\gio}[1]{{\color{black}#1}}

\newtheorem{teo}{Theorem}[section]
\newtheorem{Corollary}[teo]{Corollary}
\newtheorem{Lemma}[teo]{Lemma}
\newtheorem{Theorem}[teo]{Theorem}
\newtheorem{Proposition}[teo]{Proposition}
\theoremstyle{definition}
\newtheorem{Definition}[teo]{Definition}

\newtheorem{remark}[teo]{Remark}

\title[Higher order boundary Harnack principle via degenerate equations]
{Higher order boundary Harnack principle\\
via degenerate equations}
\date{\today}

\author{Susanna Terracini, Giorgio Tortone and Stefano Vita}

\address[Susanna Terracini]{Dipartimento di Matematica Giuseppe Peano
\newline\indent
Universit\`a degli Studi di Torino
\newline\indent
Via Carlo Alberto 10, 10124, Torino, Italy}
\email{susanna.terracini@unito.it}
\address[Giorgio Tortone]{Dipartimento di Matematica
\newline\indent
Universit\`a di Pisa
\newline\indent
Largo Bruno Pontecorvo 5, 56127, Pisa, Italy}
\email{giorgio.tortone@dm.unipi.it}
\address[Stefano Vita]{Dipartimento di Matematica Giuseppe Peano
\newline\indent
Universit\`a degli Studi di Torino
\newline\indent
Via Carlo Alberto 10, 10124, Torino, Italy}
\email{stefano.vita@unito.it}


\thanks{{\it Acknowledgment:} G.T. is supported by the European Research Council's (ERC) project n.853404 ERC VaReg - {\it Variational approach to the regularity of the free boundaries}, funded by the program Horizon 2020. The authors are research fellow of Istituto Nazionale di Alta Matematica INDAM group GNAMPA}

\begin{document}
\maketitle

\vspace{-1cm}
\begin{abstract}
As a first result we prove higher order Schauder estimates for solutions to singular/degenerate elliptic equations of type:
\begin{equation*}\label{evenrho}
-\mathrm{div}\left(\rho^aA\nabla w\right)=\rho^af+\mathrm{div}\left(\rho^aF\right) \quad\textrm{in}\; \Omega
\end{equation*}
for exponents $a>-1$, where the weight $\rho$ vanishes with non zero gradient on a regular hypersurface $\Gamma$, which can be either a part of the boundary of $\Omega$ or mostly contained  in its interior. As an application, we extend such estimates to the ratio $v/u$ of two solutions to a second order elliptic equation in divergence form when the zero set of $v$ includes the zero set of $u$ which is not singular in the domain (in this case $\rho=u$, $a=2$ and $w=v/u$).
We prove first $C^{k,\alpha}$-regularity of the ratio from one side of the regular part of the nodal set of $u$ in the spirit of the higher order boundary Harnack principle in \cite{DesSav}. Then, by a gluing Lemma, the estimates extend across the regular part of the nodal set.
Finally, using conformal mapping in dimension $n=2$, we provide local gradient estimates for the ratio which hold also across the singular set.

\end{abstract}


\section{Introduction and main results}
Let us consider a regular hypersurface $\Gamma$ embedded in $\R^n$ with $n\geq2$ and  a weight $\rho$ vanishing on it with non zero gradient. At first, this paper deals with higher order local Schauder estimates up to $\Gamma$ for solutions to singular/degenerate elliptic equations of type
\begin{equation}\label{evenrho}
-\mathrm{div}\left(\rho^aA\nabla w\right)=\rho^af+\mathrm{div}\left(\rho^aF\right)
\end{equation}
in a bounded domain $\Omega$ and with the exponent $a>-1$.  The hypersurface $\Gamma$ can be either contained in the boundary  $\partial \Omega$ or in the interior of $\Omega$. In the first case, we shall use the notation $\Omega^+$ to emphasise that $\Omega$ lies on one side of $\Gamma$. In any case, solutions have to be understood in the energy sense, as elements of the weighted Sobolev space $H^{1,a}(\Omega)=H^1(\Omega,\rho^adz)$ which satisfy
	
\begin{equation*}
\int_{\Omega}\rho^aA\nabla w\cdot\nabla\phi=\int_{\Omega}\rho^a(f\phi-F\cdot\nabla\phi),
\end{equation*}
for any test function belonging to $H^{1,a}(\Omega)$. In other words, solutions are critical points of the functional

\begin{equation*}
\int_{\Omega}\frac 12\rho^aA\nabla w\cdot\nabla w- \int_{\Omega}\rho^a(fw-F\cdot\nabla w),
\end{equation*}
which is well defined in the energy space under suitable assumptions on the right hand sides $f,F$. Note that elements of the energy space have a trace on $\Gamma$ only when $a\in(-1,1)$ (the case of $A_2$-\emph{Muckenhoupt weights}) whereas in the \emph{superdegenerate case $a\geq 1$} the space $C^{\infty}_c(\overline\Omega\setminus\Gamma)$ is dense in $H^{1,a}(\Omega)$, so traces on $\Gamma$ are meaningless. Formally, if it exists, the limit satisfies

\begin{equation}\label{evenBC0}
\lim_{\rho(z)\to0^+}\rho^a(A\nabla w+F)\cdot\nu=0,
\end{equation}
where $\nu$ is the outward unit normal vector on $\Gamma$. Thus we are associating with the equation its natural boundary condition at $\Gamma$. It is worthwhile noticing that, when $a\geq 1$ and $\Gamma$ lies inside the domain, solutions can be discontinuous at $\Gamma$ (see \gio{\cite[Example 1.4]{SirTerVit1}}).\\
\gio{In the $A_2$-Muckenhoupt case $a\in (-1,1)$, such equations have been extensively studied in the literature due to their implications for the regularity theory of fractional problems. Indeed, the relationship between degenerate operators and fractional problems has been thoroughly explored
since the work of Molchanov-Ostrovskii \cite{MO}, where the connection between L\'evy stable
processes (fractional processes) and Bessel processes (degenerate processes) was examined from a
probabilistic perspective. Subsequently, Caffarelli-Silvestre \cite{CafSil1} introduced an alternative analytical
methodology for investigating local properties of fractional equations as traces of elliptic PDEs
with degenerate coefficients. Furthermore, Chang-Gonz\'alez \cite{ChaGon} highlighted the close relationship
between the degenerate equation introduced in \cite{CafSil1} and a specific class of operators in conformal
geometry, highlighting the profound connection between scattering theory on conformally compact Einstein manifolds and conformally invariant objects on their boundaries \cite{GraZwo}. For a further
geometric interpretation of the equations studied in our work, we refer to the theory of elliptic
operators with edge degeneracies developed in \cite{Maz,MazVer}. On the other hand, Our most powerful
motivation for the study of such equations lies in application to the analysis of the ratio of solutions
to elliptic PDEs. Indeed, let $u, v$ be two solutions of the elliptic equations of type
\begin{equation*}\label{uv}
-\mathrm{div}\left(A\nabla u\right)=g\;, \quad-\mathrm{div}\left(A\nabla v\right)=f \quad\textrm{in}\; \Omega^+
\end{equation*}
such that $u\equiv v\equiv0$ on $\Gamma\subset\partial\Omega^+$. Then, one easily sees that the ratio $w:=v/u$  satisfies equation
\eqref{evenrho} with $\rho=u$ and $a=2$, for a suitable right hand side depending on $u, v, f, g$. We stress that the Schauder regularity of the ratio $v/u$ is usually referred as \emph{higher order boundary Harnack principle} and has been studied with a different approach in \cite{DesSav,BanGar,Kuk}. Finally, we would like to refer also to some recent interesting works on operators which are degenerate or singular on higher
co-dimensional sets, see \cite{DavFenMay,DavMay1}.
}

\subsection*{One-sided Schauder estimates up to the characteristic manifold $\Gamma$}

Let  us consider the upper side of a regular hypersurface $\Gamma$ embedded in $\R^n$, and localize the problem on a ball centered in $0$ which lies on $\Gamma$. Thus
\begin{equation}\label{localizeGamma}
\Omega_\varphi^+\cap B_1=\{y>\varphi(x)\}\cap B_1,\qquad \Gamma\cap B_1=\{y=\varphi(x)\}\cap B_1,
\end{equation}
with $z=(x,y)\in\R^{n-1}\times\R$, $\varphi(0)=0$ and $\nabla_x\varphi(0)=0$. In other words, we are locally describing the upper side of the manifold as the epigraph of a function $\varphi$ defined in $B_1'=B_1\cap\{y=0\}$ and the manifold as its graph. For $z\in \Gamma$, let us denote by $\nu^+(z)$ the outward unit normal vector to $\Omega_\varphi^+$. Let us consider a weight term $\rho(z)$ satisfying the following properties:
\begin{equation}\label{rhoconditions}
\begin{cases}
\rho>0 &\mathrm{in} \ \Omega_\varphi^+\cap B_1,\\
\rho=0 &\mathrm{on} \ \Gamma\cap B_1,\\
\partial_{\nu^+}\rho<0 &\mathrm{on} \ \Gamma\cap B_1.
\end{cases}
\end{equation}
Our first result concerns the validity of the Schauder estimates up to the hypersurface $\Gamma$ for energy solutions to \eqref{evenrho} in $\Omega_\varphi^+$. We extend here some results previously obtained in \cite{SirTerVit1,SirTerVit2} under an additional structural assumption on the matrix $A$ near $\Gamma$. Indeed, in the present paper the variable coefficient matrix $A=(a_{ij})_{ij}$ is only assumed to be continuous, symmetric and uniformly elliptic; that is, for some $0<\lambda\leq\Lambda<+\infty$
\begin{equation}\label{UE}
\lambda|\xi|^2\leq A(z)\xi\cdot\xi\leq\Lambda|\xi|^2
\end{equation}
for any $\xi\in\R^{n}$, and $z$ taken in the relevant bounded domain.

\begin{Theorem}[Schauder estimates up to the characteristic hypersurface $\Gamma$]\label{curvedschauder}
Let $a>-1$, $a^+=\max\{a,0\}$, $p>n+a^+$, $k\in\mathbb N$ and
\begin{equation*}
\begin{cases}
\alpha\in(0,1-\frac{n+a^+}{p}] &\mathrm{if \ }k=0 \ \mathrm{and \ } p<+\infty\\
\alpha\in(0,1) &\mathrm{if \ }k=0 \ \mathrm{and \ } p=\infty \ \mathrm{or \ }k\geq1.
\end{cases}
\end{equation*}
Let $\varphi\in C^{k+1,\alpha}(B'_1)$ and $\rho\in C^{k+1,\alpha}(\Omega_\varphi^+\cap B_1)$ satisfying \eqref{rhoconditions}. Let $w\in H^1(\Omega_\varphi^+\cap B_1,\rho(z)^adz)$ be a weak solution to \eqref{evenrho} with $F,A\in C^{k,\alpha}(\Omega_\varphi^+\cap B_1)$, $f\in L^p(\Omega_\varphi^+\cap B_1,\rho(z)^adz)$ when $k=0$ or $f\in C^{k-1,\alpha}(\Omega_\varphi^+\cap B_1)$ when $k\geq1$. Then, there exists $0<r<1$ such that $w\in C^{k+1,\alpha}(\Omega_\varphi^+\cap B_r)$ with boundary condition
\begin{equation}\label{extremalityintro}
(A\nabla w+F)\cdot\nu^+=0\qquad\mathrm{on \ } \Gamma\cap B_r.
\end{equation}
Moreover, if $\beta>1$,
$$\|A\|_{C^{k,\alpha}(\Omega_\varphi^+\cap B_1)}+\|\rho\|_{C^{k+1,\alpha}(\Omega_\varphi^+\cap B_1)}+\|\varphi\|_{C^{k+1,\alpha}(B'_1)}\leq L_1,\qquad \inf_{B_{\frac{1+r}{2}}\cap\Gamma}|\partial_{\nu^+}\rho|\geq L_2>0,
$$
then there exists a constant $c>0$ depending only on $n,\lambda,\Lambda,a,p,\alpha,r,k,\beta,L_1,L_2$, such that, for any energy solution to \eqref{evenrho} in $\Omega_\varphi^+\cap B_1$ holds
\begin{equation*}
\|w\|_{C^{k+1,\alpha}(\Omega_\varphi^+\cap B_r)}\leq c\left(\|w\|_{L^\beta(\Omega_\varphi^+\cap B_1,\rho(z)^adz)}+\|f\|_{C^{k-1,\alpha}(\Omega_\varphi^+\cap B_1)}+\|F\|_{C^{k,\alpha}(\Omega_\varphi^+\cap B_1)}\right),
\end{equation*}
where $\|f\|_{C^{k-1,\alpha}(\Omega_\varphi^+\cap B_1)}$ must be replaced with $\|f\|_{L^{p}(\Omega_\varphi^+\cap B_1,\rho(z)^adz)}$ in the case $k=0$.
\end{Theorem}

Note the appearance of a true Neumann boundary condition \eqref{extremalityintro} and compare with the natural one in \eqref{evenBC0}.

As a further consequence of Theorem \ref{curvedschauder}, when $a>0$ we will also provide Schauder estimates for solutions to the singularly forced equation
\begin{equation}\label{eqcurved1}
-\mathrm{div}\left(\rho^aA\nabla w\right)=\rho^a\dfrac{f}{\rho}
\end{equation}
in $\Omega_\varphi^+\cap B_1$ (see Theorem \ref{singularcurved}).

\subsection*{One-sided higher order boundary Harnack principle with right hand sides}
Let us denote by $L_A$ the differential operator $\mathrm{div}\left(A\nabla \cdot\right)$. Theorem \ref{curvedschauder} has a remarkable application to higher order boundary Harnack inequalities for ratios of solutions to equations ruled by $L_A$, as  obtained in \cite{DesSav} by De Silva and Savin and recently extended in \cite{Kuk} for the analogue parabolic problem in a nondivergence form. In \cite{MaiTorVel} the authors addressed the same problem in domains arising from shape optimization problems. Recent literature on boundary Harnack principles in Lipschitz domains also includes
   \cite{AllSha, RosTor}, where the presence of a right hand side is covered. Let us consider two functions $u,v$ solving
\begin{equation}\label{BHconditions}
\begin{cases}
-L_Av=f &\mathrm{in \ }\Omega_\varphi^+\cap B_1,\\
-L_Au=g &\mathrm{in \ }\Omega_\varphi^+\cap B_1,\\
u>0 &\mathrm{in \ }\Omega_\varphi^+\cap B_1,\\
u=v=0, \qquad \partial_{\nu^+} u<0&\mathrm{on \ }\Gamma\cap B_1.
\end{cases}
\end{equation}
It can be easily proven that the ratio $w=v/u$ is an energy solution to
\begin{equation}\label{e:r}
-\mathrm{div}\left(u^2A\nabla w\right)=u (f-gw) \quad\mathrm{in \ }\Omega_\varphi^+\cap B_1.
\end{equation}
\begin{Theorem}[Higher order boundary Harnack principle]\label{HOBHP}
Let $k\in\mathbb N$ and $\alpha\in(0,1)$. Let us consider two functions $u,v$ solving \eqref{BHconditions}. Let us assume that $A,f,g\in C^{k,\alpha}(\Omega_\varphi^+\cap B_1)$ and $\varphi\in C^{k+1,\alpha}(B'_1)$. Then, $w=v/u$ belongs to $C^{k+1,\alpha}_\loc(\overline{\Omega_\varphi^+}\cap B_1)$ and satisfies the following boundary condition
\begin{equation}\label{boundaryHOBH}
2(\nabla u\cdot\nu^+)A\nabla w\cdot\nu^++f-gw=0\qquad\mathrm{on \ } \Gamma\cap B_1.
\end{equation}
Moreover, if $\|A\|_{C^{k,\alpha}(\Omega_\varphi^+\cap B_{1})}+\|\varphi\|_{C^{k+1,\alpha}(B'_{1})}+\|g\|_{C^{k,\alpha}(\Omega_\varphi^+\cap B_{1})}\leq L_1, \|u\|_{L^2(\Omega_\varphi^+\cap B_{1})}\leq L_2$ and $\inf_{B_{3/4}\cap\Gamma}|\partial_{\nu^+}u|\geq L_3>0$, then the following estimate holds true
\begin{equation*}\label{eq.daperfez}
\left\|\frac{v}{u}\right\|_{C^{k+1,\alpha}(\Omega_\varphi^+\cap B_{1/2})}\leq C\left(\left\|v\right\|_{L^2(\Omega_\varphi^+\cap B_{1})}+\|f\|_{C^{k,\alpha}(\Omega_\varphi^+\cap B_{1})}\right)
\end{equation*}
for any $u,v$ satisfying \eqref{BHconditions} and with a positive constant $C$ depending on $n,\lambda,\Lambda,\alpha,k,L_1,L_2,L_3$. Finally, if $u(\vec e_n/2)=1$ and $v>0$ in $\Omega_\varphi^+\cap B_{1}$, then
\begin{equation*}
\left\|\frac{v}{u}\right\|_{C^{k+1,\alpha}(\Omega_\varphi^+\cap B_{1/2})}\leq C\left(\left|\frac{v}{u}(\vec e_n/2)\right|+\|f\|_{C^{k,\alpha}(\Omega_\varphi^+\cap B_{1})}\right)
\end{equation*}
with a positive constant $C$ depending only on $n,\lambda,\Lambda,\alpha,k,L_1,L_2,L_3$. 
\end{Theorem}
Under the assumptions of Theorem \ref{HOBHP}, the standard Schauder estimates up to the boundary imply that both $u$ and $v$ belong to the class $C^{k+1,\alpha}_\text{loc}(\overline{\Omega^+_\varphi}\cap B_1)$ (and this is optimal). Therefore, as $u$ vanishes with a nonzero gradient, we can readily deduce that $v/u$ is in $C^{k,\alpha}_\text{loc}(\overline{\Omega^+_\varphi}\cap B_1)$. Thus our result improves the basic regularity by one degree of differentiability at the common zero set, due to the fact that both $u$ and $v$ satisfy the same type of elliptic equations, also in the non homogenous case.  Note that our Theorem \ref{HOBHP} does not follow directly from the statement in \cite{DesSav} which does not allow forcings $g$, neither from \cite{Kuk}, due to the possible lack of regularity of the derivatives of the metric $A$, since the result there is proved for equations in nondivergence form. Finally, it is worthwhile stressing that it seems to be the first time that boundary Harnack estimates are derived directly from the associated superdegenerate equation.

\subsection*{Ratios of solutions to uniformly elliptic equations sharing the same zero set}
Another interesting application of our results and methods concerns the regularity of ratios of $L_A$-harmonic functions \emph{across} their common nodal set.  This is connected with Logunov and Malinnikova's theorem stating analiticity of ratios of harmonic functions  sharing the same nodal set (see \cite{LogMal1,LogMal2}). While the approach there strongly relies on division Lemmata and analiticity estimates, the authors wonder whether an alternative one could be conducted through the analysis of the associated superdegenerate equation \eqref{eqw} fullfilled by the ratio $w=v/u$ (\S 5.2 of \cite{LogMal1}). A main goal of this paper is indeed to answer positively, though partially, this question. Concerning the boundary Harnack inequality for $L_A$-harmonic functions at the common nodal set, we quote the recent work by Lin and Lin \cite{LinLin}, who prove $C^{0,\alpha}$ bounds for  the ratio $w$, for a small $\alpha$ depending on the frequency bounds on $u$ and $v$.

To begin with, let $n\geq 2$ and $u \in H^1(B_1)$ be a weak solution to
\begin{equation}\label{equv}
L_Au=0 \qquad \mathrm{in \ }B_1,
\end{equation}
where $A(z)=(a_{ij}(z))_{ij}$ is a symmetric uniformly elliptic matrix with $\alpha$-H\"older continuous coefficients for some $\alpha\in(0,1)$.

By standard Schauder theory, any weak solution is actually of class $C^{1,\alpha}_\loc(B_1)$.
Thus the nodal set $Z(u)=u^{-1}(\{0\})$ of $u$ splits into a regular part $R(u)$ and a singular part $S(u)$ defined as
\begin{equation}\label{nodalset}
R(u)=\{z\in Z(u) \ : \ |\nabla u|\neq0\},\qquad S(u)=\{z\in Z(u) \ : \ |\nabla u|=0\};
\end{equation}
where $R(u)$ is in fact locally a $(n-1)$-dimensional hypersurface of class $C^{1,\alpha}$. In general, the hypersurface $R(u)$ inherits the regularity of the associated solution $u$, by implicit function theorem. Let us fix here a notation: given $\alpha\in(0,1]$, by $u\in C^{k,\alpha-}$ we mean $u\in C^{k,\beta}$ for any $0<\beta<\alpha$. Let us remark here that if we assume additionally that $A\in C^{0,1}(B_1)$, then solutions enjoy the unique continuation property and
$S(u)$ has Hausdorff dimension at most $(n-2)$ (see e.g. \cite{Han,HanLin,GarLin}).

Given a second solution $v$  to \eqref{equv} in $B_1$ with $Z(u)\subseteq Z(v)$, it is not difficult  to prove that the ratio $w=v/u$ is in fact an energy solution to the degenerate elliptic equation
\begin{equation}\label{eqw}
\mathrm{div}(u^2A\nabla w)=0\qquad \mathrm{in \ }B_1,
\end{equation}
in the sense of weak solution belonging to the weighted Sobolev space $H^{1}(B_1,u^2dz)$. \gio{This  holds true across a wide range of coefficients and primarily relies on a Hardy-type inequality for functions vanishing on $Z(u)$ (refer to Lemma \ref{hardy}). While generic solutions to \eqref{eqw} may not necessarily be continuous at the zero set of the weight $u$,  the ratio $w$ exhibits H\"{o}lder continuity when $A$ is locally Lipschitz, as demonstrated in \cite{LinLin}.} Our next goal is to prove Schauder estimates for the ratio across the regular part $R(u)$.

\begin{Theorem}[Schauder estimates for the ratio across the regular part of the nodal set $R(u)$]\label{schauderR(u)}
Let  $A\in C^{k,\alpha}(B_1)$, for some $k\in\N$ and $\alpha\in(0,1)$, and $(u,v)$ a pair of solutions to \eqref{equv}, such that $S(u)\cap B_1=\emptyset$ and $Z(u)\subseteq Z(v)$.  Then $w=v/u\in C^{k+1,\alpha}_\loc(B_1)$ and in addition we have
\begin{equation*}
A\nabla w\cdot\nu=0\qquad \mathrm{on \ }R(u)\cap B_1,
\end{equation*}
where $\nu$ is the unit normal vector on $R(u)$. Moreover, let $\|A\|_{C^{k,\alpha}(B_1)}\leq L$ and $u$ be a solution to \eqref{equv} such that $S(u)\cap B_1=\emptyset$. Then, for any solution $v$ to \eqref{equv} with $Z(u)\subseteq Z(v)$, we have
\begin{equation*}
\left\| \frac{v}{u} \right\|_{C^{k+1,\alpha}(B_{1/2})}\leq C\left\|v\right\|_{L^2(B_1)},
\end{equation*}
with $C>0$ depending on $n,\lambda,\Lambda,\alpha,L, u$ and its nodal set $Z(u)$.
\end{Theorem}

We remark that the estimates above depend on the nodal set $Z(u)$. Indeed we postpone to the forthcoming paper \cite{TerTorVit2} the discussion about the uniformity of the constants of the estimates by varying solutions $u$ (and consequently their nodal sets $Z(u)$) in a compact class of functions with bounded frequency.

Then, our next result concerns $C^{1,\alpha}$ estimates for the ratio $w$ across $S(u)$ in dimension $n=2$. To proceed further, we assume that coefficients are Lipschitz continuous. This will be needed to classify the singular points accordingly with their vanishing order, that is for $z_0 \in Z(u)$,
\begin{equation}\label{e:vanishing}
\mathcal{V}(z_0,u) = \sup\left\{\beta\geq 0: \ \limsup_{r \to 0^+} \frac1{r^{n-1+2\beta}} \int_{\partial B_r(z_0)} u^2 <+\infty\right\}.
\end{equation}
The vanishing order $\mathcal{V}(z_0,u)\geq0$ is characterized by the property that
$$
\limsup_{r \to 0^+} \frac1{r^{n-1+2\beta}} \int_{\partial B_r(z_0)} u^2  = \begin{cases} 0 & \text{if $0 <\beta< \mathcal{V}(z_0,u)$} \\
+\infty  & \text{if $\beta > \mathcal{V}(z_0,u)$}.
\end{cases}
$$

In case of Lipschitz continuous coefficients, the detection of the vanishing orders and the structure of the nodal set are strictly related with the validity of an Almgren type monotonicity formula (see \cite{GarLin,Han,HanLin,CheNabVal}).
In dimension $n=2$ the singular set $S(u)$ is a locally finite set at which the nodal set is made of a finite number of $C^{1,1-}$ curves meeting with equal angles, being the equality of angles true when the matrix is the identity at the given singular point (see \cite{Han,Han2,HanLin,HarWin1,HarWin}). It is worthwhile noticing that, although the $C^{1,1-}$ regularity of the regular curves, far from the singularities, is a natural consequence of the implicit function theorem, its extension up to singular points is far from trivial (see Lemma \ref{lemmacurves}).

\gio{If $u$ is a solution to \eqref{equv} in $B_1$ such that $S(u)\cap B_1=\{0\}, \mathcal{V}(0,u)=N$, after composing with a linear transformation that sets $A(0)=\mathbb I$, one can observe that a connected component $\Omega_u$ of the set $\{u\neq0\}$ is asymptotically a conical domain $\Omega_{\pi/N}$, with an aperture of $\pi/N$, whose boundary $\partial\Omega_{\pi/N}$ is parameterized by the juxtaposition of two $C^{1,1-}$ curves.} Then, given $a\in \R$ such that $|u|^a\in L^1(B_1)$, one can prove the following gradient estimate for solutions to
\begin{equation}\label{eqnodal}
\mathrm{div}\left(|u|^a B\nabla w\right)=0\qquad\mathrm{in \ } \Omega_u\cap B_1.
\end{equation}
We refer to the forthcoming paper \cite{TerTorVit2} for an exhaustive discussion on the interplay between the vanishing order of $u$ and the admissible exponents $a$ for which the weight is locally integrable.

\begin{Theorem}[Gradient estimates on a nodal domain in dimension $n=2$]\label{teoBnodal}
Let $n=2$ and consider a non-trivial solution $u$ to \eqref{equv} in $B_1$ such that $A\in C^{0,1}(B_1)$ and $S(u)\cap B_1=\{0\}$, with $\mathcal{V}(0,u)=N>1$. Then any solution to \eqref{eqnodal}, with $B\in C^{0,1}(\Omega_{u}\cap B_1)$ and $ B(0)=A(0)$, belongs to $C^{1,1/N-}_\loc(\overline\Omega_{u}\cap B_1)$ and satisfies the following condition
\begin{equation*}
\begin{cases}
B\nabla w\cdot\nu=0 &\mathrm{on \ } \partial\Omega_{u}\cap B_1\\
\nabla w(0)=0.
\end{cases}
\end{equation*}
Moreover, if $\|A\|_{C^{0,1}(B_1)}+\|B\|_{C^{0,1}(\Omega_{u}\cap B_1)}\leq L_1$, $u$ is a solution to \eqref{equv} in $B_1$ with $A\in C^{0,1}(B_1)$, $S(u)\cap B_1=\{0\}$, $\|u\|_{L^2(B_1)}\leq L_2$, $\mathcal V(0,u)\leq \overline N$, $\inf_{B_{3/4}\cap\partial\Omega_\pi}|\partial_{\nu^+}\overline u|\geq L_3>0$ (where $\overline u$ is defined in \eqref{compconf}), then for every $0<\beta<1/\overline N$, the following estimate holds true
\begin{equation*}
\left\|w\right\|_{C^{1,\beta}(\Omega_{u}\cap B_{1/2})}\leq C\|w\|_{L^2(\Omega_{u}\cap B_1,|u|^a(z)dz)},
\end{equation*}
for any solution of \eqref{eqnodal} with $C$ depending only on $\lambda,\Lambda,\overline N,a,\beta,L_1,L_2,L_3$.
\end{Theorem}

\gio{In the specific case $a=2$, Theorem \ref{teoBnodal} yields $C^{1,\alpha}$ estimates for the ratio $w$ across $S(u)$. Since the vanishing order $z \mapsto \mathcal{V}(z,u)$ is upper semi-continuous  \cite[Lemma 1.4]{Han}, given $r \in (0,1)$ we can define
\begin{equation}\label{maxV}
N_0(r)=N_0(r,u)= \max_{z_0\in \overline{B_r}}\mathcal{V}(z_0,u).
\end{equation}
Clearly, the value of $N_0(r)$ depends only on the geometry of the fixed nodal set $Z(u)$.}
\begin{Theorem}[Gradient estimates for the ratio in dimension $n=2$]\label{teodim2}
Let $n=2, A\in C^{0,1}(B_1)$ and $(u,v)$ be a pair of solutions to \eqref{equv}, such that $Z(u)\subseteq Z(v)$. Then, for every $r\in (0,1)$ the ratio $w=v/u$ belongs to $C^{1,1/N_0(r)-}(B_r)$ and satisfies the following condition
\begin{equation*}
\begin{cases}
A\nabla w\cdot\nu=0 &\mathrm{on \ } R(u)\cap B_1\\
\nabla w=0 &\mathrm{on \ } S(u)\cap B_1.
\end{cases}
\end{equation*}
Moreover, let $\|A\|_{C^{0,1}(B_1)}\leq L$ and $u$ be a solution to \eqref{equv}. Then, for any solution $v$ to \eqref{equv} with $Z(u)\subseteq Z(v)$ and $0<\beta<1/N_0(1/2)$, it holds
\begin{equation*}
\left\|\frac{v}{u}\right\|_{C^{1,\beta}(B_{1/2})}\leq C\left\|v\right\|_{L^2(B_1)},
\end{equation*}
with $C>0$ depending on $n,\lambda,\Lambda,\beta, L, u$ and its nodal set $Z(u)$.
\end{Theorem}
\gio{We point out that the $C^{1,\alpha}$ estimate above is consistent with the $L^\infty$ bound given by Mangoubi \cite{Man} (always in dimension $n=2$) for gradients of ratios of harmonic functions.}

\subsection*{A Liouville type theorem}

Our technique relies upon blow-up and a Liouville type theorem, which is of independent interest. This expresses rigidity properties of entire solutions and is useful for classification purposes, upon knowledge of the growth rate at infinity. The following theorem includes \gio{\cite[Corollary 3.5]{SirTerVit1}}.

\begin{Theorem}[Liouville theorem for entire solutions on a half space]\label{liouvilletheoremzero}
Let $\rho\in L^1_\loc(\R)$ be such that \begin{enumerate}
  \item $\rho(y)>0$ for every $y>0$;
  \item there exist $a>-1$ and $C>0$ such that
  $$
  \rho(y)\leq C(1+y^a),\quad\mbox{for every }y\in [0,+\infty).
  $$
\end{enumerate}
Let $w \in H^1_\loc(\overline{\R^n_+},\rho dz)$ be a solution to
 \begin{equation}\label{evenrho0}
 \begin{cases}
 \mathrm{div}\left(\rho\nabla w\right)=0 & \mathrm{in} \ \R^n_+\\
 \lim_{y\to0^+}\rho \,\partial_yw=0 & \mathrm{on} \  \Sigma,
 \end{cases}
\end{equation}
such that for some $C,\gamma>0$
\begin{equation}\label{liouv.condition}
|w(z)|\leq C(1+|z|)^\gamma.
\end{equation}
Then if $\gamma\in[0,2)$ the function $w$ is affine and does not depend on $y$. 
Moreover, if $\gamma\in[0,1)$ then the function $w$ is constant. 
\end{Theorem}

\subsection*{Structure of the paper}
The paper is organized as follows: in Section \ref{Section2} we prove some general regularity results for solutions to elliptic equations
with coefficients which are degenerate/singular on a characteristic hyperplane or a curved characteristic manifold. Ultimately, we prove the Schauder estimate presented in Theorem \ref{curvedschauder}. Then, in Section \ref{Section3} we address two specific applications of the theory developed in Section \ref{Section2}. First, in Subsection \ref{subs:HOBH} we prove the higher order boundary Harnack principle of \cite{DesSav} through the auxiliary degenerate equation; that is, Theorem \ref{HOBHP} and then as a consequence we obtain also Theorem \ref{schauderR(u)}. Then, in Subsection \ref{sub.n2} we deal with regularity for the ratio near singular zeros in dimension $n=2$ and we prove Theorem \ref{teoBnodal} and Theorem \ref{teodim2}. Finally, in Section \ref{Section4} we prove the Liouville Theorem \ref{liouvilletheoremzero}.

\section{Schauder estimates for equations degenerating on a hypersurface}\label{Section2}
In this Section we are going to prove some regularity results for solutions to elliptic equations with coefficients which are degenerate or singular on a characteristic manifold. In what follows, avoiding some details, we will refer to the definitions and results contained in \cite{SirTerVit1,SirTerVit2}. We invite the reader who is interested in deepening the knowledge of this class of degenerate or singular equations to the reading of the papers mentioned above and the references therein.

\subsection{Gradient estimates from one side of the characteristic hyperplane}\label{sub.flat}

By a change of coordinates, we can always flatten the regular boundary $\Gamma$ and  reduce  to the case when $\Omega^+$ is the half unit ball $B_1^+=B_1\cap\{z=(x,y)\in\R^{n-1}\times\R\;,y>0\}$ and $\Gamma$ is the characteristic hyperplane  $\Sigma=\{y=0\}$. In the next Subsections we then extend the estimates for the flat case to curved manifolds. Our first achievement concerns gradient estimates up to the flat boundary for energy solutions to
\begin{equation}\label{evenLa}
-\mathrm{div}\left(y^aA\nabla u\right)=y^af+\mathrm{div}\left(y^aF\right)\qquad\mathrm{in \ }B_1^+,
\end{equation}
for $a>-1$, under a Neumann formal boundary condition as detailed earlier,
\begin{equation}\label{evenBC}
\lim_{y\to0^+}y^a(A\nabla u+F)\cdot\nu=0,
\end{equation}
where the outward unit normal vector on $\Sigma$ is $\nu=-\vec e_{n}$. Solutions to \eqref{evenLa} have to be understood as functions belonging to $H^{1,a}(B_1^+)$ satisfying the following weak formulation
\begin{equation*}
\int_{B_1^+}y^aA\nabla u\cdot\nabla\phi=\int_{B_1^+}y^a(f\phi-F\cdot\nabla\phi)
\end{equation*}
for every test function $\phi\in C^{\infty}_c(B_1)$. As we have already remarked in the Introduction, when $a\geq1$ then we can consider test functions only in $C^{\infty}_c(B_1\setminus\Sigma)$, and this is due to the strong degeneracy of the weight term in the superdegenerate setting.

\begin{Theorem}[Gradient estimates up to the characteristic hyperplane $\Sigma$]\label{EVENwithoutHA}
Let $a>-1$, $a^+=\max\{a,0\}$, $p>n+a^+$, $\alpha\in(0,1-\frac{n+a^+}{p}]$, $F,A\in C^{0,\alpha}(B_1^+)$, $f\in L^p(B_1^+,y^adz)$. Then, any energy solution $u$ to \eqref{evenLa} belongs to $C^{1,\alpha}(B_r^+)$ for any $r\in(0,1)$. Moreover, if $\|A\|_{C^{0,\alpha}(B_1^+)}\leq L$, then for $r\in(0,1)$ and $\beta>1$, there exists $C>0$ depending only on $n,\lambda,\Lambda,a,p,\alpha,r,\beta$ and $L$ such that
\begin{equation*}
\|u\|_{C^{1,\alpha}(B_r^+)}\leq C\left(\|u\|_{L^\beta(B_1^+,y^adz)}+\|f\|_{L^p(B_1^+,y^adz)}+\|F\|_{C^{0,\alpha}(B_1^+)}\right),
\end{equation*}
and the estimate is $\varepsilon$-stable in the sense of Remark \ref{epsstability}. Moreover, solutions satisfy the following boundary condition
\begin{equation}\label{extremality}
(A\nabla u+F)\cdot\vec e_{n}=0\qquad\mathrm{on} \ \Sigma.
\end{equation}
We remark also that the estimate above holds for any $\alpha\in(0,1)$ if $p=\infty$.
\end{Theorem}

Theorem \ref{EVENwithoutHA} generalizes the gradient estimates proved in \gio{\cite[Theorem 1.2]{SirTerVit1} and \cite[Theorem 1.3]{SirTerVit1}} by removing the assumption $\Sigma$ being $A$-invariant. Notice that, in order to provide gradient estimates for degenerate or singular problems from one side and up to the characteristic manifold $\Sigma=\{y=0\}$, no structural assumption on the variable coefficient matrix is needed. Moreover, one does not need to require that the last component of the field vanishes on $\Sigma$; that is,
\begin{equation}\label{fieldzero}
F(x,0)\cdot\vec e_{n}=F_{n}(x,0)=0.
\end{equation}
In fact, one can work without performing any even reflection across $\Sigma$. This is the reason why we will not refer anymore to solutions to \eqref{evenLa} using the expression \emph{even solutions}.

Here we would like to prove that actually \gio{\cite[Theorems 1.2 and 1.3]{SirTerVit1}} hold true without requiring condition \eqref{fieldzero} on the field $F$ and the structural assumption of invariance of $\Sigma$ with respect to $A$; that is, the existence of a scalar function $\mu$ such that
\begin{equation}\label{invariance}
A(x,0)\vec e_{n}=\mu(x,0)\vec e_{n}.
\end{equation}

\begin{remark}\label{epsstability}
In the statement of Theorem \ref{EVENwithoutHA}, when we say that the gradient estimate is $\varepsilon$-stable we mean that the constant in the estimate is also uniform with respect to a $\varepsilon$-perturbation of the problem with the uniformly elliptic weight terms $\rho_\varepsilon^a(y)=(\varepsilon^2+y^2)^{a/2}$ for $\varepsilon<<1$; that is, for solutions to
\begin{equation}\label{evenLaeps}
\begin{cases}
-\mathrm{div}\left(\rho_\varepsilon^aA\nabla u_\varepsilon\right)=\rho_\varepsilon^af_\varepsilon+\mathrm{div}\left(\rho_\varepsilon^aF_\varepsilon\right) &\mathrm{in \ } B_1^+\\
\rho_\varepsilon^a(A\nabla u_\varepsilon+F_\varepsilon)\cdot\vec e_{n}=0 &\mathrm{on \ } B'_1.
\end{cases}
\end{equation}
Actually, stability of the estimate for regularized problems and an approximation argument is how regularity is proven for the case $\varepsilon=0$ (for further details see \gio{\cite[Section 2]{SirTerVit1}}).
\end{remark}

In what follows we are going to prove Theorem \ref{EVENwithoutHA}, whose validity strongly relies on the Liouville Theorem \ref{liouvilletheoremzero} (see also \cite[Corollary 3.5]{SirTerVit1}).

\begin{proof}[Proof of Theorem \ref{EVENwithoutHA}]
In order to prove the result, we follows the strategies developed in the proofs of \gio{\cite[Theorems 5.1 and 5.2]{SirTerVit1}}; that is, uniform estimates for the $\varepsilon$-regularized problems, showing how to work without conditions \eqref{fieldzero} and \eqref{invariance}. Then the result holds for the case $\varepsilon=0$ by approximation working exactly as in \cite{SirTerVit1}. One can proceed by a contradiction argument  inspired by \cite{SoaTer10} in order to prove the $\varepsilon$-uniform $C^{1,\alpha}$ estimates, without performing any reflection across $\Sigma$. Hence, along a sequence $\varepsilon_k\to0$ there exist $a>-1$, $p>n+a^+$, $\beta>1$, $\alpha\in (0,1-\frac{n+a^+}{p}]$, $r\in(0,1)$ and a sequence of solutions $u_k=u_{\varepsilon_k}$ to \eqref{evenLaeps} such that uniformly
$$\|u_k\|_{L^\beta(B_1^+,\rho_{\varepsilon_k}^a(y)dz)}+\|f_{\varepsilon_k}\|_{L^p(B_1^+,\rho_{\varepsilon_k}^a(y)dz)}+\|F_{\varepsilon_k}\|_{C^{0,\alpha}(B_1^+)}\leq c$$
and
$$\|u_k\|_{C^{1,\alpha}(B_r^+)}\to+\infty.$$
In other words, one has that, for a radially decreasing cut-off function $\eta\in C^\infty_c(B_1)$ with $0\leq\eta\leq1$, $\eta\equiv1$ in $B_r$ and $\mathrm{supp}\eta=B_{\frac{1+r}{2}}=:B$, for two sequences of points $z_k,\zeta_k\in B\cap\{y\geq0\}$ and a partial derivative $i\in\{1,...,n\}$
\begin{equation*}
\frac{|\partial_i(\eta u_k)(z_k)-\partial_i(\eta u_k)(\zeta_k)|}{|z_k-\zeta_k|^\alpha}=L_k\to+\infty.
\end{equation*}
Now we want to define two blow up sequences: let $r_k=|z_k-\zeta_k|$, $\frac{1+r}{2}<\overline r<1$ and
\begin{equation*}
v_k(z)=\frac{\eta(\hat z_k+r_kz)}{L_kr_k^{1+\alpha}}\left(u_k(\hat z_k+r_kz)-u_k(\hat z_k)\right),\qquad w_k(z)=\frac{\eta(\hat z_k)}{L_kr_k^{1+\alpha}}\left(u_k(\hat z_k+r_kz)-u_k(\hat z_k)\right),
\end{equation*}
for $z\in B(k):=\frac{B^+_{\overline r}-\hat z_k}{r_k}$ and $\hat z_k\in B\cap\{y\geq0\}$ to be determined. In any case $\hat z_k=(x_k,\hat y_{k})$ where $z_k=(x_k,y_{k})$. Let $B(\infty)=\lim_{k\to+\infty}B(k)$. There are two different cases; that is,
\begin{itemize}
\item[{\bf Case 1 :}] the term $\frac{y_{k}}{r_k}$ is unbounded, then we choose $\hat z_{k}=z_k$;
\item[{\bf Case 2 :}] the term $\frac{y_{k}}{r_k}$ is bounded, then we choose $\hat y_{k}=0$. In other words, $\hat z_k=(x_k,0)$.
\end{itemize}
In {\bf Case 1}, since points $z_k$ lie on a compact set, then we already know that $r_k\to0$. The fact that $r_k\to0$ in {\bf Case 2} has to be proved after suitably adjusting the blow-up sequences: this helps also to have the sequence of derivatives uniformly bounded at a point, in order to apply the Ascoli-Arzel\'a convergence theorem. Such adjustment consists in subtracting a linear term
$$\overline v_k(z)=v_k(z)-\nabla v_k(0)\cdot z,\qquad\overline w_k(z)=w_k(z)-\nabla w_k(0)\cdot z.$$
Avoiding further details, the adjusted blow-up sequences converge both on compact subsets of the limit blow-up set to the same entire profile $w$ which is non constant, has non constant gradient and is globally $C^{1,\alpha}$ (actually the convergence on compact sets is in $C^{1,\gamma}$ for $\gamma<\alpha$). Moreover, $\overline w_k$ solves the following rescaled equations, for any $C^\infty_c$-function of the blow-up domain
\begin{eqnarray*}
&&\int_{\mathrm{supp}\phi}\rho_{\varepsilon_k}^a(\hat y_k+r_ky)A(\hat z_k)\nabla\overline w_k(z)\cdot\nabla\phi(z)=\frac{\eta(\hat z_k)}{L_kr_k^\alpha}r_k\int_{\mathrm{supp}\phi}\rho_{\varepsilon_k}^a(\hat y_k+r_ky)f_{\varepsilon_k}(\hat z_k+r_kz)\phi(z)\\
&&-\frac{\eta(\hat z_k)}{L_kr_k^\alpha}\int_{\mathrm{supp}\phi}\rho_{\varepsilon_k}^a(\hat y_k+r_ky)[F_{\varepsilon_k}(\hat z_k+r_kz)-F_{\varepsilon_k}(\hat z_k)]\cdot\nabla\phi(z)\\
&&+\int_{\mathrm{supp}\phi}\rho_{\varepsilon_k}^a(\hat y_k+r_ky)[A(\hat z_k)-A(\hat z_k+r_kz)]\nabla w_k(z)\cdot\nabla\phi(z)\\
&&-\frac{\eta(\hat z_k)}{L_kr_k^\alpha}\int_{\mathrm{supp}\phi}\rho_{\varepsilon_k}^a(\hat y_k+r_ky)[A(\hat z_k)\nabla u_k(\hat z_k)+F_{\varepsilon_k}(\hat z_k)]\cdot\nabla\phi(z).
\end{eqnarray*}
In {\bf Case 1} the limiting blow-up domain is $\R^{n}$ and hence $\mathrm{supp}\phi\subset B_R$ for some $R>0$. In {\bf Case 2} instead, the limiting blow-up domain is $\R^{n}_+$ and hence $\mathrm{supp}\phi\subset B_R\cap \{y\geq0\}$. The first term in the right hand side vanishes using the uniform boundedness of $f_{\varepsilon_k}$ in the suitable weighed Lebesgue spaces, while the second and third terms vanish using boundedness of coefficients and field terms in $C^{0,\alpha}$ (for the third term see also \cite[Remark 5.3]{SirTerVit1}). In order to prove that also the last term vanishes, let us remark that
\begin{eqnarray*}
&&\int_{\mathrm{supp}\phi}\rho_{\varepsilon_k}^a(\hat y_k+r_ky)[A(\hat z_k)\nabla u_k(\hat z_k)+F_{\varepsilon_k}(\hat z_k)]\cdot\nabla\phi(z)=\\
&&\int_{\mathrm{supp}\phi}\mathrm{div}\left(\rho_{\varepsilon_k}^a(\hat y_k+r_ky)[A(\hat z_k)\nabla u_k(\hat z_k)+F_{\varepsilon_k}(\hat z_k)]\phi(z)\right)\\
&&-\int_{\mathrm{supp}\phi}\partial_y\left(\rho_{\varepsilon_k}^a(\hat y_k+r_ky)\right)[A(\hat z_k)\nabla u_k(\hat z_k)+F_{\varepsilon_k}(\hat z_k)]\cdot\vec e_{n}\phi(z).
\end{eqnarray*}
Hence, using the divergence theorem, the first term in the right hand side is zero in both {\bf Case 1} and {\bf Case 2}. In the second case, since $\hat z_k$ lies on the flat boundary, then
\begin{equation*}
A(\hat z_k)\nabla u_k(\hat z_k)+F_{\varepsilon_k}(\hat z_k)\cdot\vec e_{n}=0.
\end{equation*}
So, also the second term is identically zero in {\bf Case 2}. In order to conclude, we have to deal with the term
$$\frac{\eta(z_k)}{L_kr_k^\alpha}\int_{\mathrm{supp}\phi}\partial_y\left(\rho_{\varepsilon_k}^a(y_k+r_ky)\right)[A(z_k)\nabla u_k(z_k)+F_{\varepsilon_k}(z_k)]\cdot\vec e_{n}\phi(z)$$
in {\bf Case 1} ($\hat z_k=z_k=(x_k,y_k)$). Let us define $\xi_k=P_\Sigma(z_k)=(x_k,0)$ the projections on $\Sigma$. The term $|\nu_k^{-a}\partial_y(\rho_{\varepsilon_k}^a(\hat y_k+r_ky))|$ can be controlled from above on compact sets by $Cr_k/y_k$. Then, adding and subtracting $\eta(\xi_k)[A(\xi_k)\nabla u_k(\xi_k)+F_{\varepsilon_k}(\xi_k)]=0$ one has
$$|\eta A\nabla u_k(z_k)-\eta A\nabla u_k(\xi_k)|\leq |A\nabla(\eta u_k)(z_k)- A\nabla(\eta u_k)(\xi_k)|+|u_k A\nabla \eta(z_k)-u_k A\nabla \eta(\xi_k)|\leq CL_ky_k^\alpha.$$
Eventually, one can estimate the full term as
\begin{equation*}
\left|\frac{\eta(z_k)\nu_k^{-a}}{L_kr_k^\alpha}\int_{\mathrm{supp}\phi}\partial_y(\rho_{\varepsilon_k}^a(\hat y_k+r_ky))[A(\hat z_k)\nabla u_k(\hat z_k)+F_{\varepsilon_k}(\hat z_k)]\cdot\vec e_{n}\phi(z)dz\right|\leq C\left(\frac{r_k}{y_k}\right)^{1-\alpha}\to0.
\end{equation*}
Then, in {\bf Case 1} one can prove that the limit is an entire solution to
\begin{equation*}
\mathrm{div}\left(\hat A\nabla w\right)=0 \qquad\mathrm{in \ } \R^{n}
\end{equation*}
and in {\bf Case 2} to
\begin{equation*}
-\mathrm{div}\left(\rho(y) \hat A\nabla w\right)=0 \qquad\mathrm{in \ } \R^{n}_+=\{y>0\}
\end{equation*}
with ``homogeneous Neumann boundary condition" on $\Sigma$, in the sense that $w\in H^1_\loc(\overline{\R^{n}_+},\rho(y)dz)$ solves weakly
\begin{equation*}
\int_{\R^{n}_+}\rho(y)\hat A\nabla w\cdot\nabla\phi=0
\end{equation*}
for all test functions $\phi\in C^\infty_c(\R^{n})$ (when $a\geq1$ test functions can be taken in $C^\infty_c(\R^{n}\setminus\Sigma)$).
The weight term $\rho(y)$ is either $1$, $y^a$ or $(1+y^2)^{a/2}$, and the limit matrix $\hat A$ possesses constant coefficients and is symmetric positive definite. Let us consider the square root $C=\hat A^{1/2}$ of $\hat A$, which is symmetric and positive definite too. We remark that the linear transformation associated to the inverse of such a matrix maps $\R^{n}$ in itself, and in case the blow-up limit set is $\R^{n}_+=\{y>0\}$, then it maps such half space in another half space. In fact, with the change of variable $C \zeta=z$, then
$$\{y>0\}=\{z\cdot \vec e_{n}>0\}=\{C\zeta\cdot \vec e_{n}>0\}=\{\zeta\cdot C\vec e_{n}>0\},$$
which is a half space also in the new coordinate system, with boundary hyperplane given by $\Sigma'=\{\zeta\cdot C\vec e_{n}=0\}$. In other words, the new outward normal vector is related to the old one by the following formula $\nu'(\zeta)=C\nu(C \zeta)$, and in our case is the constant vector $\nu'=-C\vec e_{n}$.

Let $v(\zeta)=w(C \zeta)$. Then, up to dilations, the function $v$ is a weak solution of
\begin{equation*}
-\mathrm{div}\left(\rho'(\zeta)\nabla v\right)=0 \qquad\mathrm{in \ } \{\zeta\cdot C\vec e_{n}>0\}
\end{equation*}
with ``homogeneous Neumann boundary condition" on $\Sigma'$, in the sense that $v$ belongs to the space $H^1_\loc(\overline{\{\zeta\cdot C\vec e_{n}>0\}},\rho'(\zeta)d\zeta)$ and solves weakly
\begin{equation*}
\int_{\{\zeta\cdot C\vec e_{n}>0\}}\rho'(\zeta)\nabla v\cdot\nabla\phi=0
\end{equation*}
for all test functions $\phi\in C^\infty_c(\R^{n})$  (when $a\geq1$ test functions can be taken in $C^\infty_c(\R^{n}\setminus\Sigma')$), where $\rho'(\zeta)$ is either $1$, $\mathrm{dist}(\zeta,\Sigma')^a$ or $(1+\mathrm{dist}(\zeta,\Sigma')^2)^{a/2}$. The global $C^{1,\alpha}$ regularity implies a subquadratic growth at infinity
$$|v(z)|\leq |v(z)-v(0)-\nabla v(0)\cdot z|+|v(0)|+|\nabla v(0)|\, |z|\leq C(1+|z|)^{1+\alpha}.$$
Hence, one can now apply the suitable Liouville type theorem which brings to a contradiction; that is, the classical Liouville theorem for harmonic functions with polynomial growth in {\bf Case 1} or Theorem \ref{liouvilletheoremzero} in {\bf Case 2}. In any case, $v$ should be a linear function, in contradiction with the fact that it possesses non constant gradient.

Finally, let us remark that the boundary condition \eqref{extremality} comes from the local $C^{1}$ convergence of the regularized solutions to the limit one in the approximation scheme.
\end{proof}

\subsection{Schauder estimates from one side of the characteristic hyperplane}\label{sub2.2}
Aim of the Section is the iteration of the gradient estimate in Theorem \ref{EVENwithoutHA} on the derivatives of the solution. The regularity for tangential derivatives follows by differentiation in the equation, while the regularity for the normal derivative is based on some technical Lemmata contained in what follows.

\subsubsection{Some preliminary results}
The following Lemmata will be crucial for the Schauder estimates in Subsection \ref{subs.Scha}.
\begin{Lemma}\label{lem1}
Let $k\in\mathbb N$, $\alpha\in(0,1]$. Let $f\in C^{k+1,\alpha}(B_1)$ with $f(x,0)=0$. Then $f/y\in C^{k,\alpha}(B_1)$.
\end{Lemma}
\proof
By direct computation
$$f(x,y)=\int_0^y\partial_yf(x,t)dt=y\int_0^1\partial_yf(x,sy)ds.$$
Hence
\begin{equation*}
\frac{f(x,y)}{y}=\int_0^1\partial_yf(x,sy)ds
\end{equation*}
and regularity follows by Leibniz rule. Indeed, if we consider a multiindex $|\beta|=k$, then
\begin{align*}
|D^\beta(f/y)(x_1,y_1)-D^\beta(f/y)(x_2,y_2)|&\leq \int_0^1|D^\beta\partial_yf(x_1,sy_1)-D^\beta\partial_yf(x_2,sy_2)|ds\\
&\leq C|(x_1,y_1)-(x_2,y_2)|^\alpha.
\end{align*}
\endproof

\begin{Lemma}\label{lem2}
Let $a>-1$, $k\in\mathbb N$ and $\alpha\in(0,1]$. Let $g\in C^{k,\alpha}(B_1)$. Then
\begin{equation*}
\varphi(x,y)=\frac{1}{y|y|^a}\int_0^y|t|^ag(x,t)dt
\end{equation*}
belongs to $C^{k,\alpha}(B_1)$.
\end{Lemma}
\proof
We proceed by induction. In the case $k=0$, one can argue as in the proof of \cite[Lemma 7.5]{SirTerVit1} in order to get $\varphi\in C^{0,\alpha}(B_1^+)$ and $\varphi\in C^{0,\alpha}(B_1^-)$. Hence, it remains to prove that $\varphi$ is actually continuous across $\Sigma$, which is true since
$$\lim_{y\to0+}\varphi(x_0,y)=\lim_{y\to0-}\varphi(x_0,y)=\frac{g(x_0,0)}{a+1}.$$
Let us assume that the result is true for a generic integer $k\in\mathbb N$ and let us prove the result for $k+1$. In other words, we claim that $\partial_{x_i}\varphi\in C^{k,\alpha}(B_1)$, for any $i=1,...,n-1$, and $\partial_y\varphi\in C^{k,\alpha}(B_1)$. Here we are assuming $g\in C^{k+1,\alpha}(B_1)$, hence
$$\partial_{x_i}\varphi(x,y)=\frac{1}{y|y|^a}\int_0^y|t|^a\partial_{x_i}g(x,t)dt$$
which belongs to $C^{k,\alpha}(B_1)$ by inductive hypothesis. Then, let us express the function $\varphi$ as
\begin{equation*}
\varphi(x,y)=\frac{1}{y|y|^a}\int_0^y|t|^a(g(x,t)-g(x,0))dt+\frac{g(x,0)}{a+1}.
\end{equation*}
Hence,
\begin{equation*}
\partial_y\varphi(x,y)=-\frac{a+1}{y|y|^{a+1}}\int_0^y|t|^{a+1}\frac{g(x,t)-g(x,0)}{t}dt+\frac{g(x,y)-g(x,0)}{y}.
\end{equation*}
By Lemma \ref{lem1}, $(g(x,y)-g(x,0))/y\in C^{k,\alpha}(B_1)$ and hence we can conclude using again the inductive hypothesis.
\endproof

The previous Lemma immediately implies the following \begin{remark}\label{lem3}
Let $a>-1$, $k\in\mathbb N$ and $\alpha\in(0,1]$. Let $g\in C^{k,\alpha}(B_1)$. Then
\begin{equation*}
\varphi(x,y)=\frac{1}{|y|^a}\int_0^y|t|^ag(x,t)dt
\end{equation*}
possesses partial derivative $\partial_y\varphi\in C^{k,\alpha}(B_1)$.
In fact
$$\partial_y\varphi(x,y)=-\frac{a}{|y|^ay}\int_0^y|t|^ag(x,t)dt+g(x,y).$$
\end{remark}

\subsubsection{Schauder estimates}\label{subs.Scha}

Aim of this Subsection is to prove the following result

\begin{Theorem}[Schauder estimates up to the characteristic hyperplane $\Sigma$]\label{EVENwithoutHA1}
Let $a>-1$, $k\in\mathbb N\setminus\{0\}$, $\alpha\in(0,1)$, $F,A\in C^{k,\alpha}(B_1^+)$, $f\in C^{k-1,\alpha}(B_1^+)$. Then, any energy solution $u$ to \eqref{evenLa} belongs to $C^{k+1,\alpha}(B_r^+)$ for any $r\in(0,1)$. Moreover, if $\|A\|_{C^{k,\alpha}(B_1^+)}\leq L$, then for $r\in(0,1)$ and $\beta>1$, there exists $C>0$ depending only on $n,\lambda,\Lambda,a,k,\alpha,r,\beta$ and $L$ such that
\begin{equation*}
\|u\|_{C^{k+1,\alpha}(B_r^+)}\leq C\left(\|u\|_{L^\beta(B_1^+,y^adz)}+\|f\|_{C^{k-1,\alpha}(B_1^+)}+\|F\|_{C^{k,\alpha}(B_1^+)}\right).
\end{equation*}
Moreover, the solutions satisfy the boundary condition \eqref{extremality} on $\Sigma$.
\end{Theorem}

For the sake of simplicity, we split the proof of Theorem \ref{EVENwithoutHA1} into Lemma \ref{teok} and Lemma \ref{teok1}.

In Lemma \ref{teok}, for $a>-1$, we are going to prove Schauder estimates for solutions to the following problem
\begin{equation}\label{evenLa3}
-\mathrm{div}\left(y^aA\nabla u\right)=\mathrm{div}\left(y^aF\right)
\end{equation}
in $B_1^+$, where $A,F$ belong to H\"older spaces. Solutions to \eqref{evenLa3} must be understood as functions belonging to $H^{1,a}(B_1^+)$ satisfying the following weak formulation
\begin{equation*}
-\int_{B_1^+}y^aA\nabla u\cdot\nabla\phi=\int_{B_1^+}y^aF\cdot\nabla\phi
\end{equation*}
for every test function $\phi\in C^{\infty}_c(B_1)$ (when $a\geq1$ test functions can be taken in $C^{\infty}_c(B_1\setminus\Sigma)$).

Actually, with Lemma \ref{teok1}, we will improve also the Schauder estimates in \cite[\S 7]{SirTerVit1}; that is, for solutions to
\begin{equation}\label{evenLa31}
-\mathrm{div}\left(y^aA\nabla u\right)=y^af
\end{equation}
in $B_1^+$, in the sense of $H^{1,a}(B_1^+)$-functions such that
\begin{equation*}
\int_{B_1^+}y^aA\nabla u\cdot\nabla\phi=\int_{B_1^+}y^af\phi
\end{equation*}
for every test function $\phi\in C^{\infty}_c(B_1)$ (when $a\geq1$ test functions can be taken in $C^{\infty}_c(B_1\setminus\Sigma)$).

\begin{Lemma}\label{teok}
Let $a>-1$, $k\in\mathbb N$, $\alpha\in(0,1)$ and $u\in H^{1,a}(B_1^+)$ be a weak solution to \eqref{evenLa3} with $F,A\in C^{k,\alpha}(B_1^+)$. Then, $u\in C^{k+1,\alpha}(B_r^+)$ for any $0<r<1$ and satisfies \eqref{extremality}. Moreover, if $\|A\|_{C^{k,\alpha}(B_1^+)}\leq L$, then for $\beta>1$ and $0<r<1$, there exists a constant $c>0$, depending on $n,\lambda,\Lambda,a,\alpha,r,k,\beta$ and $L$, such that
\begin{equation*}
\|u\|_{C^{k+1,\alpha}(B_r^+)}\leq c\left(\|u\|_{L^\beta(B_1^+,y^adz)}+\|F\|_{C^{k,\alpha}(B_1^+)}\right).
\end{equation*}
\end{Lemma}
\begin{proof}
We proceed by induction. The result in case $k=0$ is true by Theorem \ref{EVENwithoutHA}. Let us assume the result true for $k\in\mathbb N$ and prove it for $k+1$. Thus, we are assuming $A,F\in C^{k+1,\alpha}(B_1^+)$ and we would like to prove that $\partial_{x_i}u,\partial_yu\in C^{k+1,\alpha}(B_r^+)$ for any $i=1,...,n-1$.
Notice that the tangential derivatives $\partial_{x_i}u$ solve
\begin{equation*}\label{evenLa3x}
-\mathrm{div}\left(y^aA\nabla (\partial_{x_i}u)\right)=\mathrm{div}\left(y^a(\partial_{x_i}F+\partial_{x_i}A\nabla u)\right) \qquad\mathrm{in \ } B_1^+,
\end{equation*}
in the sense that
\begin{equation*}
-\int_{B_1^+}y^aA\nabla (\partial_{x_i}u)\cdot\nabla\phi=\int_{B_1^+}y^a(\partial_{x_i}F+\partial_{x_i}A\nabla u)\cdot\nabla\phi
\end{equation*}
for every test function $\phi\in C^{\infty}_c(B_1)$ (when $a\geq1$ test functions can be taken in $C^{\infty}_c(B_1\setminus\Sigma)$).

Since the field $\partial_{x_i}F+\partial_{x_i}A\nabla u$ belongs to $C^{k,\alpha}(B_r^+)$, by inductive hypothesis we have
\begin{equation}\label{xireg}
\partial_{x_i}u\in C^{k+1,\alpha}(B_r^+)\qquad \mathrm{for \  any \ } i=1,...,n-1.
\end{equation}

Now, equation \eqref{evenLa3} can be rewritten as
\begin{equation*}
-\mathrm{div}\left(A\nabla u\right)=\frac{a(A\nabla u+F)\cdot\vec e_{n}}{y}+\mathrm{div}F
\end{equation*}
Hence
\begin{equation}\label{eqphi}
\partial_y\varphi+\frac{a}{y}\varphi=g:=-\mathrm{div}F+\partial_yF_{n}-\sum_{i=1}^{n-1}\partial_{x_i}(A\nabla u\cdot \vec e_i),
\end{equation}
where $g\in C^{k,\alpha}(B_r^+)$ and $\varphi:=(A\nabla u+F)\cdot\vec e_{n}$ with $\varphi(x,0)=0$. Since $\partial_y\varphi+\frac{a}{y}\varphi=y^{-a}\partial_y\left(y^a\varphi\right)$, then one would like to prove that
\begin{equation}\label{phireg}
\varphi(x,y)=\frac{1}{y^a}\int_0^yt^ag(x,t)dt\in C^{k+1,\alpha}(B_r^+).
\end{equation}
Eventually, this last information together with \eqref{xireg} would give $u\in C^{k+2,\alpha}(B_r^+)$. In fact, isolating the term $\partial^2_{yy}u$ in the left hand side of \eqref{eqphi}, one gets the desired regularity also for this last derivative from the equation once one observes that the uniform ellipticity of $A$ \eqref{UE} gives the uniform bound from below
\begin{equation}\label{ellipticn}
a_{n,n}(x,y)=A(x,y)\vec e_{n}\cdot\vec e_{n}\geq\lambda>0.
\end{equation}
In order to prove \eqref{phireg} it is enough to remark that \eqref{xireg} and the definition of $\varphi$ immediately give $\partial_{x_i}\varphi\in C^{k,\alpha}(B_r^+)$. Then, by Remark \ref{lem3} one also gets $\partial_y\varphi\in C^{k,\alpha}(B_r^+)$.
\end{proof}

The result above implies the following

\begin{Lemma}\label{teok1}
Let $a>-1$, $k\in\mathbb N$, $\alpha\in(0,1)$ and $u\in H^{1,a}(B_1^+)$ be a weak solution to \eqref{evenLa31} with $f\in C^{k,\alpha}(B_1^+)$, $A\in C^{k+1,\alpha}(B_1^+)$. Then, $u\in C^{k+2,\alpha}(B_r^+)$ for any $0<r<1$ and satisfies
\begin{equation}\label{extremality3}
A\nabla u\cdot\vec e_{n}=0\qquad \mathrm{on \ }\Sigma.
\end{equation}
Moreover, if $\|A\|_{C^{k+1,\alpha}(B_1^+)}\leq L$, then for  $\beta>1$ and $0<r<1$, there exists a constant $c>0$ depending on $n,\lambda,\Lambda,a,\alpha,r,k,\beta$ and $L$ such that
\begin{equation*}
\|u\|_{C^{k+2,\alpha}(B_r^+)}\leq c\left(\|u\|_{L^\beta(B_1^+,y^adz)}+\|f\|_{C^{k,\alpha}(B_1^+)}\right).
\end{equation*}
\end{Lemma}
\proof
We proceed by induction. Let $k=0$. Then we already know that \eqref{extremality3} holds and $u\in C^{1,\alpha}(B_r^+)$ by Theorem \ref{EVENwithoutHA}. We would like to prove that actually $\partial_{x_i}u,\partial_y u\in C^{1,\alpha}(B_r^+)$ for any $i=1,...,n-1$. The tangential derivatives solve the equation
\begin{equation*}
-\mathrm{div}\left(y^aA\nabla(\partial_{x_i}u)\right)=\mathrm{div}\left(y^a(\partial_{x_i}A\nabla u+f\vec e_i)\right),
\end{equation*}
whose solutions belong to $C^{1,\alpha}(B_r^+)$ again by Theorem \ref{EVENwithoutHA} since the field $\partial_{x_i}A\nabla u+f\vec e_i$ belongs to $C^{0,\alpha}$. Then, in order to prove that $\partial_y u\in C^{1,\alpha}(B_r^+)$ it remains to prove that $\partial^2_{yy}u\in C^{0,\alpha}(B_r^+)$ (being $\partial_{x_i}\partial_yu\in C^{0,\alpha}(B_r^+)$ already implied by $\partial_{x_i}u\in C^{1,\alpha}(B_r^+)$). Proceeding as in the proof of Lemma \ref{teok}, starting from the equation \eqref{evenLa31}, one can obtain the expression \eqref{phireg} for $\varphi=A\nabla u\cdot\vec e_{n}$ with $g=-f-\sum_{i=1}^{n-1}\partial_{x_i}(A\nabla u\cdot\vec e_i)\in C^{0,\alpha}$. Hence, by Lemma \ref{rem1}, we get $\partial_y\varphi\in C^{0,\alpha}$, which together with the condition \eqref{ellipticn} gives $\partial^2_{yy}u\in C^{0,\alpha}(B_r^+)$. Then the induction works with the very same reasonings applying Lemma \ref{teok}.
\endproof

\begin{proof}[Proof of Theorem \ref{EVENwithoutHA1}]
Follows by Lemmata \ref{teok} and \ref{teok1}.
\end{proof}

As a further consequence of Theorem \ref{EVENwithoutHA1} (actually Lemma \ref{teok}), when $a>0$ we can also provide Schauder estimates for solutions to the singularly forced equation

\begin{equation}\label{evenLa2}
-\mathrm{div}\left(y^aA\nabla u\right)=y^a\dfrac{f}{y}
\end{equation}
in $B_1^+$, in the sense of $H^{1,a}(B_1^+)$-functions such that
\begin{equation*}
\int_{B_1^+}y^aA\nabla u\cdot\nabla\phi=\int_{B_1^+}y^a\frac{f}{y}\phi
\end{equation*}
for every test function $\phi\in C^{\infty}_c(B_1)$ (when $a\geq1$ test functions can be taken in $C^{\infty}_c(B_1\setminus\Sigma)$).

\begin{Corollary}[Schauder estimates up to the characteristic hyperplane $\Sigma$ with a singular forcing term]\label{Cork}
Let $a>0$, $k\in\mathbb N$, $\alpha\in(0,1)$, $f,A\in C^{k,\alpha}(B_1^+)$. Then, any energy solution $u$ to \eqref{evenLa2} belongs to $C^{k+1,\alpha}(B_r^+)$ for any $r\in(0,1)$. Moreover, if $\|A\|_{C^{k,\alpha}(B_1^+)}\leq L$, then for $r\in(0,1)$ and $\beta>1$, there exists $c>0$ depending only on $n,\lambda,\Lambda,a,k,\alpha,r,\beta$ and $L$ such that
\begin{equation*}
\|u\|_{C^{k+1,\alpha}(B_r^+)}\leq c\left(\|u\|_{L^\beta(B_1^+,y^adz)}+\|f\|_{C^{k,\alpha}(B_1^+)}\right).
\end{equation*}
Moreover, the solutions satisfy the following boundary condition
\begin{equation}\label{extremality2}
aA\nabla u\cdot\vec e_{n}+f=0\qquad \mathrm{on \ }\Sigma.
\end{equation}
\end{Corollary}

\begin{proof}[Proof]
One can rewrite the right hand side of \eqref{evenLa2} as a right hand side of \eqref{evenLa3} in the following way
\begin{equation*}
y^a\dfrac{f}{y}=\mathrm{div}\left(y^aF\right),
\end{equation*}
where
\begin{equation*}
F(x,y)=\vec e_{n}\frac{1}{y^{a}}\int_0^yt^{a-1}f(x,t)dt.
\end{equation*}
Thanks to Lemma \ref{lem2}, if $a>0$ and $f\in C^{k,\alpha}(B_1^+)$ then $F\in C^{k,\alpha}(B_1^+)$. Therefore, $F(x,0)\cdot\vec e_{n}=F_{n}(x,0)=\frac{f(x,0)}{a}$. Hence, the validity of Lemma \ref{teok} implies the result.
\end{proof}

\subsection{The gluing Lemma}\label{sub:glue}
In this Subsection, we show that the previous Schauder estimates hold also across the characteristic manifold, once one assumes continuity across $\Sigma$ of the solution $u$.
\begin{Lemma}[Gluing Lemma]\label{gluinglemma}
The following propositions hold true:
\begin{itemize}
\item[1)] Let $a>-1$, $p>n+a^+$, $k\in\mathbb N$ and $\alpha\in(0,1-\frac{n+a^+}{p}]$ when $k=0$ or $\alpha\in(0,1)$ when $k=0$ and $p=\infty$ or $k\geq1$. Let $u\in C(B_1)$. Let us call $u_+\in H^{1,a}(B_1^+)$, $u_-\in H^{1,a}(B_1^-)$ respectively the restrictions of $u$ to the upper and lower half balls, and let us assume that they are energy solutions to \eqref{evenLa} respectively on $B_1^+$ and on $B_1^-$. If $A,F$ belong to $C^{k,\alpha}(B_1)$, $f$ belongs to $L^p(B_1,|y|^adz)$ when $k=0$ and to $C^{k-1,\alpha}(B_1)$ when $k\geq1$, then the function $u$ belongs to $C^{k+1,\alpha}(B_r)\cap H^{1,a}(B_r)$ for any $0<r<1$ and is solution to \eqref{evenLa} in $B_r$.
\item[2)] Let $a>0$, $k\in\mathbb N$, $\alpha\in(0,1)$, $u\in C(B_1)$. Let us call $u_+\in H^{1,a}(B_1^+)$, $u_-\in H^{1,a}(B_1^-)$ respectively the restrictions of $u$ to the upper and lower half balls, and let us assume that they are energy solutions to \eqref{evenLa2} respectively on $B_1^+$ and on $B_1^-$. If $A,f$ belong to $C^{k,\alpha}(B_1)$, then the function $u$ belongs to $C^{k+1,\alpha}(B_r)\cap H^{1,a}(B_r)$ for any $0<r<1$ and is solution to \eqref{evenLa2} in $B_r$.
\end{itemize}
\end{Lemma}
\proof
Let us prove $1)$, being the second point a straightforward consequence. We would like to show by induction that, for any order $|\beta|\leq k+1$, the derivatives $D^\beta u$ glue continuously across $\Sigma$. Set $k=0$, then by Theorem \ref{EVENwithoutHA} we know that $u\in C(B_1)$ and $u_+\in C^{1,\alpha}(B_r^+), u_-\in C^{1,\alpha}(B_r^-)$. Moreover, $u$ weakly solves
\begin{equation*}
-\mathrm{div}\left(|y|^aA\nabla u\right)=|y|^af+\mathrm{div}\left(|y|^aF\right)\qquad\mathrm{in \ } B_1^+ \ \mathrm{and} \ B_1^-.
\end{equation*}
Hence, tangential derivatives are actually continuous across $\Sigma$, since
\begin{equation*}
\lim_{y\to0^+}\partial_{x_i}u(x_0,y)=\lim_{y\to0^+}\lim_{t\to0}\frac{u(x_0+t\vec e_i,y)-u(x_0,y)}{t}=\lim_{t\to0}\frac{u(x_0+t\vec e_i,0)-u(x_0,0)}{t}=l
\end{equation*}
and
\begin{eqnarray*}
\lim_{y\to0^-}\partial_{x_i}u(x_0,y)&=&\lim_{y\to0^-}\lim_{t\to0}\frac{u(x_0+t\vec e_i,y)-u(x_0,y)}{t}\\
&=&\lim_{y\to0^+}\lim_{t\to0}\frac{u(x_0+t\vec e_i,-y)-u(x_0,-y)}{t}=l.
\end{eqnarray*}
Moreover, the boundary condition \eqref{extremality} at $\Sigma$ allows to prove continuity also for the last partial derivative across the hyperplane
\begin{equation*}
\partial_y u=\frac{1}{a_{n,n}}\left(-F_{n}-\sum_{i=1}^{n-1}a_{n,i}\partial_{x_i}u\right)\qquad \mathrm{on \ }B'_1.
\end{equation*}
Now we suppose the result true for a generic integer $k\in\mathbb N$ and we prove it for $k+1$ by arguing as in the proofs of Lemmata \ref{teok} and \ref{teok1}. Indeed, any tangential derivative $\partial_{x_i}u$ is actually solution to the suitable problem obtained after differentiation on both $B_1^+$ and $B_1^-$, and enjoys the regularity across the hyperplane by inductive hypothesis. Then, in order to obtain the regularity for the last partial derivative $\partial_y u$, one uses again the equations and the properties for function $\varphi=(A\nabla u+F)\cdot\vec e_{n}$ obtained thanks to Lemmata \ref{lem1} and \ref{lem2} and Remark \ref{lem3}.

Now, we would like to prove that actually $u$ belongs to $H^{1,a}(B_r)$. Notice that when $a\in(-1,1)$, the weight is $A_2$-Muckenhoupt and hence the $(H=W)$ property holds true (see \cite[Remark 2.3]{SirTerVit1}); that is, in order to belong to the relevant Sobolev space it is enough to have a finite weighted norm. Indeed, $u\in H^{1,a}(B_r)$ follows trivially in the case of $u \in C^{1,\alpha}$ and a locally integrable weight. On the other hand, when the weight is superdegenerate; that is, $a\geq1$, then the $(H=W)$ property does not necessarily hold, but one can find $C^{\infty}_c(\overline B_1\setminus\Sigma)$-functions arbitrarily close to $u$ in the $H^{1,a}$-norm by juxtaposition of a couple in $\left(C^{\infty}_c(\overline{B_1^+}\setminus\Sigma),C^{\infty}_c(\overline{B_1^-}\setminus\Sigma)\right)$ which are arbitrarily close in the norm of the upper and lower half balls to $u_+$ and $u_-$. This fact is due to the strong degeneracy of the weight term and the density of smooth functions with compact support in $\overline{B_1}\setminus\Sigma$ in the Sobolev space.
Finally, the fact that $u$ is solution of \eqref{evenLa} in $B_r$ is a trivial consequence of the weak formulations for $u_+$ and $u_-$ in the half balls $B_1^+$ and $B_1^-$.
\endproof

\subsection{Curved characteristic manifolds}\label{sec-curved}
The goal of this Subsection is to extend the results of Subsections \ref{sub.flat}, \ref{sub2.2} and \ref{sub:glue} to elliptic problems which are degenerate/singular on a curved characteristic manifold. Consider now a regular hypersurface $\Gamma$ embedded in $\R^{n}$ for $n\geq2$ (where the variable in $\R^{n}$ is denoted by $z=(x,y)\in\R^{n-1}\times\R$) and localize the problem on a ball with center on the characteristic manifold (for simplicity the origin). Hence, in such a ball we can describe the domain which lives from one side of the manifold as in \eqref{localizeGamma}; that is,
\begin{equation*}
\Omega_\varphi^+\cap B_1=\{y>\varphi(x)\}\cap B_1\qquad\mathrm{with}\qquad \Gamma\cap B_1=\{y=\varphi(x)\}\cap B_1.
\end{equation*}
In other words, we are locally describing the upper side of the manifold as the epigraph of a function $\varphi$ and the manifold as its graph.
For $z\in \Gamma$, let us denote by $\nu^+(z)$ the outward unit normal vector to $\Omega_\varphi^+$. Let us consider a weight term $\rho(z)$ satisfying the properties in \eqref{rhoconditions}. For $a>-1$, let us consider weak solutions $w\in H^1(\Omega_\varphi^+\cap B_1,\rho(z)^adz)$ to
\begin{equation}\label{eqcurved}
-\mathrm{div}\left(\rho^aA\nabla w\right)=\rho^af+\mathrm{div}\left(\rho^aF\right)
\end{equation}
in $\Omega_\varphi^+\cap B_1$ in the sense that
\begin{equation*}
\int_{\Omega_\varphi^+\cap B_1}\rho^aA\nabla w\cdot\nabla\phi=\int_{\Omega_\varphi^+\cap B_1}\rho^a(f\phi-F\cdot\nabla\phi)
\end{equation*}
for any $\phi\in C^\infty_c(B_1)$ (when $a\geq1$ test functions can be taken in $C^{\infty}_c(B_1\setminus\Sigma)$). The main result of the Subsection is Theorem \ref{curvedschauder}.

\begin{proof}[Proof of Theorem \ref{curvedschauder}]
The result is a direct consequence of Theorem \ref{EVENwithoutHA} and Theorem \ref{EVENwithoutHA1}, after applying a classical diffeomorphism which straighten the boundary; that is,
\begin{equation}\label{diffeo}
\Phi(x,y)=(x,y+\varphi(x))=(\overline x,\overline y),
\end{equation}
which is of class $C^{k+1,\alpha}$. There exists a small enough $R>0$ such that
$\Phi(B_R\cap\{y>0\})\subseteq B_1\cap\{y>\varphi(x)\}$; that is, it is a subset of the domain where the original equation is satisfied and $\Phi(0)=\Phi^{-1}(0)=0$. Additionally, the boundary $B_R\cap\{y=0\}$ is mapped into the boundary $B_1\cap\{y=\varphi(x)\}$.
The Jacobian associated with $\Phi$ is given by
\begin{align*}
J_\Phi(x)=\left(\begin{array}{c|c}
\mathbb{I}_{n-1}&{\mathbf 0}\\\hline
(\nabla \varphi(x))^T&1
\end{array}\right), \qquad \mathrm{with}\quad |\mathrm{det} \, J_\Phi(x)|\equiv1.
\end{align*}
Up to a possible dilation, one can translate the study of \eqref{eqcurved} into the study of the following problem for $\tilde w=w\circ\Phi$
\begin{equation*}
-\mathrm{div}\left(\tilde\rho^a\tilde A\nabla \tilde w\right)=\tilde\rho^a\tilde f+\mathrm{div}\left(\tilde\rho^a \tilde F\right)\qquad\mathrm{in \ }  B_1^+,
\end{equation*}
where $\tilde\rho=\rho\circ\Phi$, $\tilde f=f\circ\Phi$, $\tilde F=J_\Phi^{-1}F\circ\Phi$ and
\begin{equation}\label{tildeA}
\tilde A=(J_\Phi^{-1}) (A \circ\Phi) (J_\Phi^{-1})^T.
\end{equation}
The new weight term $\tilde\rho$ belongs to $C^{k+1,\alpha}(B_1^+)$ and satisfies
\begin{equation*}\label{tilderhoconditions}
\tilde\rho>0 \quad\mathrm{in} \  B_1^+,\qquad
\tilde\rho=0 \quad\mathrm{on} \ B'_1\quad\mbox{and}\quad
\partial_y\tilde\rho>0 \quad\mathrm{on} \  B'_1.
\end{equation*}
Hence, the conditions above together with Lemma \ref{lem1}, assure that
\begin{equation*}
\frac{\tilde\rho}{y}\in C^{k,\alpha}(B_1^+)\qquad\mathrm{and}\qquad \frac{\tilde\rho}{y}\geq\mu>0 \qquad\mathrm{in \ } B_1^+\cup B'_1,
\end{equation*}
implying eventually that $\tilde w$ is solution to
\begin{equation*}
-\mathrm{div}\left(y^a \overline A\nabla \tilde w\right)=y^a\overline f+\mathrm{div}\left(y^a \overline F\right)\qquad\mathrm{in \ }  B_1^+,
\end{equation*}
with $\overline A=\tilde A(\tilde\rho/y)^a\in C^{k,\alpha}(B_1^+)$ and uniformly elliptic, $\overline F=\tilde F(\tilde\rho/y)^a\in C^{k,\alpha}(B_1^+)$ and $\overline f=\tilde f(\tilde\rho/y)^a$ belongs to $L^p(B_1^+,y^adz)$ when $k=0$ and to $C^{k-1,\alpha}(B_1^+)$ when $k\geq1$. Moreover, the new outward unit normal vector is
\begin{equation}\label{normaltransf}
-\vec e_n=(J_{\Phi})^T(\nu^+\circ\Phi)^T\sqrt{1+|\nabla_x\varphi|^2}.
\end{equation}
This can be checked having that
\begin{equation*}
\nu^+\circ\Phi(x,0)=\frac{(\nabla_x\varphi(x),-1)}{\sqrt{1+|\nabla_x\varphi(x)|^2}}.
\end{equation*}
As we have already remarked, regularity for $\tilde w$ follows from Theorem \ref{EVENwithoutHA} and Theorem \ref{EVENwithoutHA1}. Then, the regularity is inherited by $w$ through composition with the diffeomorphism $w=\tilde w\circ\Phi^{-1}$.
\end{proof}

As in Corollary \ref{Cork}, when $a>0$ we can provide regularity also for weak solutions $w\in H^1(\Omega_\varphi^+\cap B_1,\rho(z)^adz)$ to
\begin{equation}\label{eqcurved1}
-\mathrm{div}\left(\rho^aA\nabla w\right)=\rho^a\dfrac{f}{\rho}
\end{equation}
in $\Omega_\varphi^+\cap B_1$ in the sense that
\begin{equation*}
-\int_{\Omega_\varphi^+\cap B_1}\rho^aA\nabla w\cdot\nabla\phi=\int_{\Omega_\varphi^+\cap B_1}\rho^a\dfrac{f}{\rho}\phi
\end{equation*}
for any $\phi\in C^\infty_c(B_1)$ (when $a\geq1$ test functions can be taken in $C^{\infty}_c(B_1\setminus\Sigma)$).
\begin{Theorem}\label{singularcurved}
Let $a>0$, $k\in\mathbb N$, $\alpha\in(0,1)$. Let $\varphi\in C^{k+1,\alpha}(B'_1)$ and $\rho\in C^{k+1,\alpha}(\Omega_\varphi^+\cap B_1)$ satisfying \eqref{rhoconditions}. Let $w\in H^1(\Omega_\varphi^+\cap B_1,\rho(z)^adz)$ be a weak solution to \eqref{eqcurved1} with $f,A\in C^{k,\alpha}(\Omega_\varphi^+\cap B_1)$. Then, there exists $0<r<1$ such that $u\in C^{k+1,\alpha}(\Omega_\varphi^+\cap B_r)$ with boundary condition
\begin{equation}\label{bcaf}
a(\nabla\rho\cdot\nu^+)A\nabla w\cdot\nu^++f=0\qquad\mathrm{on \ } \Gamma\cap B_r.
\end{equation}
Moreover, if $$
\|A\|_{C^{k,\alpha}(\Omega_\varphi^+\cap B_1)}+\|\rho\|_{C^{k+1,\alpha}(\Omega_\varphi^+\cap B_1)}+\|\varphi\|_{C^{k+1,\alpha}(B'_1)}\leq L_1,\quad \mbox{and}\quad \inf_{B_{\frac{1+r}{2}}\cap\Gamma}|\partial_{\nu^+}\rho|\geq L_2>0,
$$
the for $\beta>1$, there exists $c>0$ depending on $n,\lambda,\Lambda,a,\alpha,r,k,\beta,L_1$, and $L_2$ such that
\begin{equation*}
\|w\|_{C^{k+1,\alpha}(\Omega_\varphi^+\cap B_r)}\leq c\left(\|w\|_{L^\beta(\Omega_\varphi^+\cap B_1,\rho(z)^adz)}+\|f\|_{C^{k,\alpha}(\Omega_\varphi^+\cap B_1)}\right).
\end{equation*}
\end{Theorem}
\begin{proof}
\gio{After composing with the diffeomorphism $\Phi$ defined in \eqref{diffeo}, the result follows by applying Corollary \ref{Cork}. Hence, in what follows we just show how the boundary condition \eqref{bcaf} on the curved manifold is derived from the one \eqref{extremality2} on the hyperplane.} Let us indicate the variable on the straightened domain with $z=(x,y)$ and the one on the curved domain with $\overline z=(\overline x,\overline y)$. Then $\Phi(z)=\overline z$ and $\Phi^{-1}(\overline z)=z$ can be read as
\begin{equation*}
\begin{cases}
x=\overline x\\
y=\overline y-\varphi(\overline x)
\end{cases}\qquad\qquad
\begin{cases}
\overline x=x\\
\overline y=y+\varphi(x).
\end{cases}
\end{equation*}
The boundary condition \eqref{extremality2} in Corollary \ref{Cork} implies the boundary condition
$$a\overline A(z)\nabla \tilde w(z)\cdot\vec e_n+\overline f(z)=0\qquad\mathrm{on \ }\Sigma.$$
Then, from the latter and \eqref{normaltransf}, given a point $z=(x,y)$ it follows
$$-a\sqrt{1+|\nabla_x\varphi(x)|^2}\left(\lim_{y\to0^+}\frac{\rho\circ\Phi(z)}{y}\right) (A\circ\Phi(x,0))[(\nabla w)\circ\Phi(x,0)]\cdot(\nu^+\circ\Phi(x,0))+ f\circ\Phi(x,0)=0,$$
which leads to \eqref{bcaf}. In fact, defining $H(\overline z)=\overline y-\varphi(\overline x)$ with $\nabla H(\overline z)=(-\nabla_{\overline x}\varphi(\overline x),1)=-\nu^+(\overline z)\sqrt{1+|\nabla_{\overline x}\varphi(\overline x)|^2}$ on $\Gamma$, we have
\begin{equation*}
\lim_{y\to0^+}\frac{\rho\circ\Phi(z)}{y}=\lim_{t\to0^-}\frac{\rho(\overline z+t\nu^+(\overline z))}{H(\overline z+t\nu^+(\overline z))}=\frac{\nabla\rho\cdot\nu^+(\overline z)}{\nabla H\cdot\nu^+(\overline z)}=-\frac{\nabla\rho\cdot\nu^+(\overline z)}{\sqrt{1+|\nabla_{\overline x}\varphi(\overline x)|^2}},
\end{equation*}
where $\overline z$ lies on $\Gamma$ and is such that $\overline z+t\nu^+(\overline z)=\Phi(z)$.
\end{proof}

\gio{Finally, the gluing Lemma \ref{gluinglemma} can be extended to the case of curved characteristic manifolds.} Localizing the problem again, one describes the upper and lower sides of the hypersurface $\Gamma$ as the epigraph and hypograph of a function $\varphi$,
\begin{equation*}
\Omega_\varphi^+\cap B_1=\{y>\varphi(x)\}\cap B_1,\qquad \Omega_\varphi^-\cap B_1=\{y<\varphi(x)\}\cap B_1
\end{equation*}
with
\begin{equation*}
\Gamma\cap B_1=\{y=\varphi(x)\}\cap B_1.
\end{equation*}
For $z\in \Gamma$, let us denote by $\nu^+(z)$ the outward unit normal vector to $\Omega_\varphi^+$ and $\nu^-(z)=-\nu^+(z)$ the outward unit normal vector to $\Omega_\varphi^-$. Let us consider a weight term $\rho(z)$ satisfying
\begin{equation}\label{rhoconditions2}
\begin{cases}
\rho\neq0 &\mathrm{in} \ B_1\setminus\Gamma\\
\rho=0 &\mathrm{on} \ \Gamma\cap B_1\\
\partial_{\nu^+}\rho\neq0 &\mathrm{on} \ \Gamma\cap B_1.
\end{cases}
\end{equation}
Then, in this setting, the following gluing Lemma holds true.
\begin{Lemma}[Gluing Lemma across curved manifolds]\label{gluinglemma2}
The following propositions hold true:
\begin{itemize}
\item[1)] Let $a>-1$, $p>n+a^+$, $k\in\mathbb N$, $\alpha\in(0,1-\frac{n+a^+}{p}]$ when $k=0$ or $\alpha\in(0,1)$ when $k=0$ and $p=\infty$ or $k\geq1$. Let $\varphi\in C^{k+1,\alpha}(B'_1)$ and $\rho\in C^{k+1,\alpha}(B_1)$ satisfying \eqref{rhoconditions2}. Let $w\in C(B_1)$. Let us call $w_+\in H^{1}(\Omega_\varphi^+\cap B_1,|\rho(z)|^adz)$, $w_-\in H^{1}(\Omega_\varphi^-\cap B_1,|\rho(z)|^adz)$ respectively the restrictions of $w$ to the upper and lower sides of the manifold, and let us assume that they are energy solutions to \eqref{eqcurved} respectively in $\Omega_\varphi^+\cap B_1$ and in $\Omega_\varphi^-\cap B_1$. If $A,F$ belong to $C^{k,\alpha}(B_1)$, $f$ belongs to $L^p(B_1,|\rho(z)|^adz)$ when $k=0$ and to $C^{k-1,\alpha}(B_1)$ when $k\geq1$, then there exists $0<r<1$ such that the function $w\in C^{k+1,\alpha}(B_r)\cap H^{1}(B_r,|\rho(z)|^adz)$ and is solution to \eqref{eqcurved} in $B_r$.
\item[2)] Let $a>0$, $k\in\mathbb N$, $\alpha\in(0,1)$. Let $\varphi\in C^{k+1,\alpha}(B'_1)$ and $\rho\in C^{k+1,\alpha}(B_1)$ satisfying \eqref{rhoconditions2}. Let $w\in C(B_1)$. Let us call $w_+\in H^{1}(\Omega_\varphi^+\cap B_1,|\rho(z)|^adz)$, $w_-\in H^{1}(\Omega_\varphi^-\cap B_1,|\rho(z)|^adz)$ respectively the restrictions of $w$ to the upper and lower sides of the manifold, and let us assume that they are energy solutions to \eqref{eqcurved1} respectively in $\Omega_\varphi^+\cap B_1$ and in $\Omega_\varphi^-\cap B_1$. If $A,f$ belong to $C^{k,\alpha}(B_1)$, then there exists $0<r<1$ such that the function $w\in C^{k+1,\alpha}(B_r)\cap H^{1}(B_r,|\rho(z)|^adz)$ and is solution to \eqref{eqcurved1} in $B_r$.
\end{itemize}
\end{Lemma}
\proof
The proof relies on Lemma \ref{gluinglemma} once one flattens $\Gamma$ using the standard diffeomorphism in \eqref{diffeo}. We would like to remark that, under the conditions $\rho\in C^{k+1,\alpha}(B_1)$ and \eqref{rhoconditions2}, then
$$\left|\frac{\tilde\rho}{y}\right|^a\in C^{k,\alpha}(B_1),\qquad \left|\frac{\tilde\rho}{y}\right|\geq\mu>0,$$
since $\tilde\rho/y\neq0$ in $B_1$ and it belongs to $C^{k,\alpha}(B_1)$ by Lemma \ref{lem1}. Finally, the composition with function $|\cdot|^a$ maintains the same regularity.
\endproof

\section{Ratio of solutions to elliptic equations sharing zero sets}\label{Section3}
In this Section, we want to give an application of our Schauder theory for degenerate equations: regularity for the ratio $w=v/u$ of two solutions to \eqref{equv} satisfying $Z(u)\subseteq Z(v)$.
\subsection{The structure of the nodal set}\label{structure.nodal}
Let $n\geq 2$ and $u \in H^1(B_1)$ be a weak solution to \eqref{equv} where $A(z)=(a_{ij}(z))_{ij}$ is a symmetric uniformly elliptic matrix with H\"older continuous coefficients. In such case, the nodal set $Z(u)=u^{-1}(\{0\})$ splits into a regular part $R(u)$ and a singular part $S(u)$ (see \eqref{nodalset})
where $R(u)$ is in fact locally a $(n-1)$-dimensional hypersurface of class $C^{1,\alpha}$. If we assume additionally that the coefficient are Lipschitz continuous, then by standard elliptic regularity, any weak solution is actually of class $C^{1,1-}_\loc(B_1)$. Thus, by \cite{Han} (see also \cite{HanLin,GarLin}) the regular part $R(u)$ is locally a $(n-1)$-dimensional hypersurface of class $C^{1,1-}$ and $S(u)$ has Hausdorff dimension at most $(n-2)$.\\

The Lipschitz continuity of the coefficients $A$ allows to prove the strong unique continuation principle (see \cite{GarLin}), which consists in the fact that non-trivial solutions can not vanish with infinite order at $Z(u)$ (see \eqref{e:vanishing} for the definition of vanishing order). Ultimately, it implies that non-trivial solutions can not vanish identically in any open subset of their reference domain,
which is the classical unique continuation principle. Under the Lipschitz regularity assumption, it is possible to exploit the validity of an Almgren-type monotonicity formula, which is the key tool to the local analysis of solution near their nodal set (see \cite{GarLin,Han,CheNabVal} for more details in this direction).

Moreover, in this context the Almgren monotonicity formula allows to gain compactness in the class $\mathcal{S}_{N_0}$ of solutions to \eqref{equv} with bounded frequency (see \cite{HanLin2,LinLin,NabVal}). We postpone to the forthcoming paper \cite{TerTorVit2} any discussion on uniformity of the estimates in the compact class $\mathcal S_{N_0}$, which allows to get regularity not depending on the variety $Z(u)$.

Let us recall the notion of vanishing order of solutions to \eqref{equv} in case of Lipschitz coefficients given in \eqref{e:vanishing}; that is, for $z_0 \in Z(u)$
\begin{equation*}
\mathcal{V}(z_0,u) = \sup\left\{\beta\geq 0: \ \limsup_{r \to 0^+} \frac1{r^{n-1+2\beta}} \int_{\partial B_r(z_0)} u^2 <+\infty\right\}.
\end{equation*}
The number $\mathcal{V}(z_0,u)\geq0$ is characterized by the property that
$$
\limsup_{r \to 0^+} \frac1{r^{n-1+2\beta}} \int_{\partial B_r(z_0)} u^2  = \begin{cases} 0 & \text{if $0 <\beta< \mathcal{V}(z_0,u)$} \\
+\infty  & \text{if $\beta > \mathcal{V}(z_0,u)$}.
\end{cases}
$$
Then, as we have remarked in \eqref{maxV}, since the vanishing order $z \mapsto \mathcal{V}(z,u)$ is upper semi-continuous \cite[Lemma 1.4]{Han}, given $r \in (0,1)$ we can define
\begin{equation*}
N_0(r)=N_0(r,u)= \max_{z_0\in \overline{B_r}}\mathcal{V}(z_0,u).
\end{equation*}

\subsection{The equation for the ratio}\label{sec:3.2}
Now, assume that $u,v$ are two solutions to \eqref{equv} satisfying $Z(u)\subseteq Z(v)$. First, one would like to prove that the ratio $w=v/u$ is in fact an energy solution to the degenerate elliptic equation \eqref{eqw} in $B_1$; that is,
\begin{equation*}
\mathrm{div}(u^2A\nabla w)=0\qquad \mathrm{in \ }B_1.
\end{equation*}
Direct computations show that $w$ is a solution in a classical sense to \eqref{eqw} in $B_1\setminus Z(u)$ (when coefficients are Lipschitz continuous and $Z(u)$ has empty interior then the equation holds also almost everywhere in $B_1$). In the next Subsection we give the notion of energy solutions to \eqref{eqw} by means of weighted Sobolev spaces and weak solutions. Then, the fact that the ratio is an energy solution to \eqref{eqw} relies mostly on the validity of a Hardy type inequality for functions vanishing on $Z(u)$.
We would like to remark here that the results contained in \S \ref{sec:3.2.1} and \S \ref{sec:3.2.2} hold true in great generality on coefficients $A$, since $u,v$ need just to be continuous.

\subsubsection{A Hardy inequality on $Z(u)$}\label{sec:3.2.1}
The fact that the ratio $w=v/u$ is an energy solution to equation \eqref{eqw} relies mostly on the following Hardy-type inequality for functions vanishing on the zero set of $u$.
\begin{Lemma}[Hardy-type inequality on $Z(u)$]\label{hardy} Let $u$ be super $L_A$-harmonic in $B_1$; that is, weakly solves $-\mathrm{div}(A\nabla u)\geq0$ in $B_1$. Then, for any $\phi\in C^\infty_c(B_1\setminus Z(u))$
\begin{equation}\label{hardyineq}
\int_{B_1}\frac{|\nabla u|^2}{u^2}\phi^2\leq \left(\frac{2\Lambda}{\lambda}\right)^2\int_{B_1}|\nabla\phi|^2.
\end{equation}
\end{Lemma}
\proof
Let us test inequality $-\mathrm{div}(A\nabla u)\geq0$ in $B_1$ with $\phi^2/u$. Then
\begin{eqnarray*}
\lambda\int_{B_1}\frac{|\nabla u|^2}{u^2}\phi^2 &\leq& \int_{B_1}\frac{A\nabla u\cdot\nabla u}{u^2}\phi^2\\
&\leq& 2\int_{B_1}\frac{A\nabla u\cdot\nabla \phi}{u}\phi\\
&\leq&2\left(\int_{B_1}\frac{|A\nabla u|^2}{u^2}\phi^2\right)^{1/2}\left(\int_{B_1}|\nabla\phi|^2\right)^{1/2}\\
&\leq&2\Lambda\left(\int_{B_1}\frac{|\nabla u|^2}{u^2}\phi^2\right)^{1/2}\left(\int_{B_1}|\nabla\phi|^2\right)^{1/2}.
\end{eqnarray*}
\endproof
We refer to the previous result using the expression \emph{Hardy inequality} since the term $|\nabla u|^2/u^2$ is possibly singular on $Z(u)$. For instance, near the regular part $R(u)$, it behaves as $\mathrm{dist}(z,R(u))^{-2}$. Therefore, the Hardy inequality holds true for any solution $v$ of \eqref{equv} having $Z(u)\subseteq Z(v)$. This is due to the following remark.
\begin{Lemma}\label{rem1}
Let $v\in H^1(B_1)$. Then for any $\varepsilon>0$ there exists $v_\varepsilon\in C^\infty_c(\overline B_1\setminus Z(v))$ such that $\|v-v_\varepsilon\|_{H^1(B_1)}<\varepsilon$.
\end{Lemma}
\proof
Without loss of generality we can assume that $v\in C^\infty(\overline B_1)$ by density. Let us define $\eta\in C^\infty(\R)$ with $\eta(t)=0$ for $|t|<1$ and $\eta(t)=t$ for $|t|>2$. Fixed $\varepsilon>0$, we consider $v_\varepsilon=\varepsilon\eta(v/\varepsilon)$. Hence it is easy to check that $v_\varepsilon$ strongly converges to $v$ in $H^1(B_1)$ and that the support of $v_\varepsilon$ is compactly contained in $\overline B_1\setminus Z(v)$. In fact
\begin{equation*}
\int_{B_1}|v_\varepsilon-v|^2\leq 2\int_{B_1\cap\{|v|\leq2\varepsilon\}\setminus\{v=0\}}|v|^2\to0
\end{equation*}
and
\begin{equation*}
\int_{B_1}|\nabla v_\varepsilon-\nabla v|^2\leq c\int_{B_1\cap\{|v|\leq2\varepsilon\}\setminus\{v=0,|\nabla v|=0\}}|\nabla v|^2\to0
\end{equation*}
since $|B_1\cap\{|v|\leq2\varepsilon\}\setminus\{v=0\}|\to0$ and $|B_1\cap\{|v|\leq2\varepsilon\}\setminus\{v=0,|\nabla v|=0\}|\to0$ being $B_1\cap\{v=0,|\nabla v|\neq 0 \}$ a set with Hausdorff dimension at most $(n-1)$. 
\endproof

\subsubsection{Energy solutions in weighted Sobolev spaces}\label{sec:3.2.2}
Fixed $u$ a solution to \eqref{equv} in $B_1$, let us define the weighted Sobolev space $H^1(B_1,u^2dz)$ as the completion of $C^\infty(\overline B_1)$ with respect to the norm
\begin{equation}\label{wnorm}
\|w\|_{H^1(B_1,u^2dz)}=\left(\int_{B_1}u^2 w^2+\int_{B_1}u^2 |\nabla w|^2\right)^{1/2}.
\end{equation}
We would like to remark that the weight $u^2$ does not belong to the $A_2$-Muckenhoupt class and is \textsl{superdegenerate} on $Z(u)$ since the vanishing order of $u$ is at least 1. An important consequence of the strong degeneracy is contained in the following useful remark.
\begin{remark}\label{rem2}
The space $H^1(B_1,u^2dz)$ is actually the completion of $C^\infty_c(\overline B_1\setminus Z(u))$ with respect to the weighted norm \eqref{wnorm}. Indeed, consider $w\in C^\infty(\overline B_1)$ with finite $H^1(B_1,u^2dz)$-norm. Then, we can construct a function arbitrarily close to $w$ in the $H^1(B_1,u^2dz)$-norm which belongs to $C^\infty_c(\overline B_1\setminus Z(u))$. Let, for $\delta>0$
\begin{equation}\label{fdelta}
f_\delta(t)=\begin{cases}
0 &\mathrm{if} \  |t|\leq\delta^2\\
\log(|t|/\delta^2)/\log(1/\delta) &\mathrm{if} \ \delta^2< |t|<\delta\\
1 &\mathrm{if} \ |t|\geq\delta.
\end{cases}
\end{equation}
Then, by choosing $\delta$ small enough, $w_\delta=wf_\delta(u)$ is arbitrarily close to $w$ in the $H^1(B_1,u^2dz)$-norm. Hence, one can use the uniform continuity of $u$ to gain enough space near $Z(u)$ to perform a convolution promoting $w_\delta$ to a smooth function.
\end{remark}

\begin{Definition}[Energy solution]\label{definition.energy}
A function $w\in H^1(B_1,u^2dz)$ is an energy solution to \eqref{eqw} in $B_1$ if
\begin{equation*}
\int_{B_1}u^2 A\nabla w\cdot\nabla\phi=0,
\end{equation*}
 for any $\phi\in C^\infty_c(B_1)$ (actually test functions can be taken in $C^\infty_c(B_1\setminus Z(u))$).
\end{Definition}
\begin{Proposition}[The ratio is an energy solution]\label{theratiosolution}
Let $0<r<1$ and $u,v$ be two solutions to \eqref{equv} in $B_1$ with $Z(u)\subseteq Z(v)$. Then, the ratio $w=v/u$ belongs to $H^1(B_r,u^2dz)$ and is an energy solution to \eqref{eqw} in $B_r$.
\end{Proposition}
\proof
The proof is an adaptation of \gio{\cite[Proposition 2.10]{SirTerVit2}}. We would like to remark that in this setting the $(H=W)$ property does not necessarily hold (again this is due to the fact that the weight does not belong to $A_2$, see \cite[Remark 2.3]{SirTerVit1}). Hence, fixed $r<1$, we have to construct a $C^\infty_c(\overline B_r\setminus Z(u))$-candidate which is arbitrarily close to $w$ in the weighted norm. For $\delta>0$, and given $\eta\in C^\infty_c(B_1)$ with $\eta\equiv 1$ in $B_r$ and $0\leq\eta\leq 1$, let us define $\varphi_\delta=\eta f_\delta(u)$ where $f_\delta$ is defined in \eqref{fdelta}.
First, we prove that $w_\delta=\varphi_\delta w$ are uniformly bounded in $H^1(B_1,u^2dz)$ with respect to $\delta$, and they converge to $w$ in $B_r$. Since the approximations have compact support away from $Z(u)$, it is enough to prove the finiteness of the weighted norm in order to deduce a uniform bound in $\delta$. Indeed, by applying the Hardy inequality of Lemma \ref{hardy}, we get
\begin{equation*}
\int_{B_1}u^2|f_\delta|^2|\nabla(\eta w)|^2\leq c\left(\int_{B_1}|\nabla(\eta v)|^2+\int_{B_1}\frac{|\nabla u|^2}{u^2}(\eta v)^2\right)\leq c\int_{B_1}|\nabla(\eta v)|^2.
\end{equation*}
The validity of the Hardy inequality \eqref{hardyineq} for the function $\eta v\in H^1_0(B_1)$ is due to Lemma \ref{rem1} since $Z(u)\subseteq Z(v)$. Then, one can regularize the function $w_\delta$ by convolution and deduce that $w\in H^1(B_r,u^2dz)$. Moreover, the fact that it satisfies the weak formulation for \eqref{eqw} is trivial since the equation holds in $B_1\setminus Z(u)$ and since Remark \ref{rem2} allows us to take test functions with compact support in $B_1\setminus Z(u)$.
\endproof

\subsection{The higher order boundary Harnack principle on and across regular zero sets}
In this Subsection we would like to prove Schauder estimates up to regular boundaries and across regular zero sets for ratios of solutions to some elliptic problems. First, we show that the results in Section \ref{Section2} imply an alternative proof of the \textsl{higher order boundary Harnack principle} proved in \cite{DesSav}. Then, in the setting of \cite{LogMal1,LogMal2,LinLin}, we will prove Schauder regularity for the ratio of two solutions which share the nodal set across its regular part.

\subsubsection{The higher order boundary Harnack principle}\label{subs:HOBH}
Consider the upper side of a regular hypersurface $\Gamma$ embedded in $\R^n$, and localize the problem as in \eqref{localizeGamma}. Then, let $u,v$ be solutions to \eqref{BHconditions}; that is,
\begin{equation*}
\begin{cases}
-\mathrm{div}\left(A\nabla v\right)=f &\mathrm{in \ }\Omega_\varphi^+\cap B_1,\\
-\mathrm{div}\left(A\nabla u\right)=g &\mathrm{in \ }\Omega_\varphi^+\cap B_1,\\
u>0 &\mathrm{in \ }\Omega_\varphi^+\cap B_1,\\
u=v=0, \qquad \partial_{\nu^+} u<0&\mathrm{on \ }\Gamma\cap B_1.
\end{cases}
\end{equation*}
\gio{
Let $k\in\mathbb N$ and $\alpha\in(0,1)$, and suppose $A$, $f$, $g\in C^{k,\alpha}(\Omega_\varphi^+\cap B_1)$, and $\varphi\in C^{k+1,\alpha}(B'_1)$. Then, by standard elliptic regularity, for every $r \in (0,1)$ we have $u,v\in C^{k+1,\alpha}(\Omega_\varphi^+\cap B_r)$, and consequently, their ratio $w=v/u$ belongs to $C^{k,\alpha}(\Omega_\varphi^+\cap B_r)$. In light of the higher-order boundary Harnack principle established in \cite{DesSav}, we deduce that the ratio actually belongs to $C^{k+1,\alpha}(\Omega_\varphi^+\cap B_r)$, indicating it shares the same regularity as the boundary $\Gamma$. The results in \cite{DesSav} are derived under the condition $g=0$, eliminating the necessity of the non-degeneracy condition $\partial_{\nu^+} u<0$ on $\Gamma\cap B_1$, which is, in fact, a direct consequence of the Boundary Point Principle, also known as the Hopf-Oleinik Lemma. Although the latter result is classical, we refer to \cite{ApuNaz,KozKuz} for its validity when $k=0$, i.e., when the boundary is $C^{1,\alpha}$ and leading coefficients are $C^{0,\alpha}$.
}

Aim of this brief Subsection is to show that the Schauder estimates for degenerate equations of Section \ref{Section2} imply the higher order boundary Harnack principle of \cite{DesSav}.

\begin{proof}[Proof of Theorem \ref{HOBHP}]
Let us first consider the local diffeomorphism in \eqref{diffeo}. Then, up to dilations, $\tilde v=v\circ\Phi$ and $\tilde u=u\circ\Phi$ solve

\begin{equation*}
\begin{cases}
-\mathrm{div}\left(\tilde A\nabla \tilde v\right)=\tilde f &\mathrm{in \ }B_1^+,\\
-\mathrm{div}\left(\tilde A\nabla \tilde u\right)=\tilde g &\mathrm{in \ }B_1^+,\\
\tilde u>0 &\mathrm{in \ }B_1^+,\\
\tilde u=\tilde v=0, \qquad -\partial_y \tilde u<0&\mathrm{on \ }B'_1,
\end{cases}
\end{equation*}
where $\tilde f=f\circ\Phi$, $\tilde g=g\circ\Phi$ and
\begin{equation*}
\tilde A=(J_\Phi^{-1}) (A \circ\Phi) (J_\Phi^{-1})^T.
\end{equation*}
The Jacobian matrixes are defined in \S \ref{sec-curved}. We remark that, after the diffeomorphism, $\tilde f,\tilde g, \tilde A$ belong to $C^{k,\alpha}(B_1^+)$, and hence $\tilde u,\tilde v\in C^{k+1,\alpha}(B_r^+)$ for any $r<1$. The fact that $\tilde w=w\circ\Phi=\tilde v/\tilde u\in C^{k,\alpha}$, follows once we notice that
$$\frac{\tilde v}{\tilde u}=\frac{\tilde v}{y}\left(\frac{\tilde u}{y}\right)^{-1}.$$
Indeed, by combining that $\tilde u>0$ in $B_1^+$ with $-\partial_y \tilde u<0$ on $B'_1$, we deduce that $\tilde u/y\geq\mu>0$ up to $\Sigma$. Finally, the result follows since, by Lemma \ref{lem1}, the ratio is the product of two $C^{k,\alpha}$ functions. In order to prove that the ratio $\tilde w$ actually belongs to $C^{k+1,\alpha}$, we use the fact that is solution to
\begin{equation}\label{eqdessav}
-\mathrm{div}\left(y^2\overline A\nabla \tilde w\right)=y^2\frac{\overline f}{y},
\end{equation}
and hence Corollary \ref{Cork} implies the desired regularity, and \eqref{extremality2} implies the boundary condition
\begin{equation*}
2\overline A\nabla\tilde w\cdot\vec e_n+\overline f=0\qquad\mathrm{on \ }\Sigma.
\end{equation*}
In \eqref{eqdessav}, the new matrix and forcing term are
$$\overline A=(\tilde u/y)^2\tilde A,\qquad \overline f=(\tilde u/y)\tilde f-(\tilde v/y)\tilde g=(\tilde u/y)\left(\tilde f-\tilde g\tilde w\right)$$
which belong to $C^{k,\alpha}(B_r^+)$. Nevertheless coefficients of $\overline A$ do not vanish on $\Sigma$ and hence the matrix is still uniformly elliptic, thanks to the non degeneracy condition $\tilde u/y\geq\mu>0$ up to $\Sigma$.

In order to prove that $\tilde w$ belongs to $H^{1,2}(B_1^+)=H^{1}(B_1^+,y^2dz)$ one uses the Hardy's inequality contained in \gio{\cite[Lemma B.1]{SirTerVit2}}.

Eventually, by \eqref{normaltransf}, one obtains the validity of the boundary condition \eqref{boundaryHOBH}.
\end{proof}

\subsubsection{The gluing Lemma across regular zero sets and covering}
In this Subsection, we show that the ratio of two solutions $u,v$ of \eqref{equv} enjoys actually Schauder regularity across the regular part $R(u)$ of the nodal set, depending on the regularity of leading coefficients. In other words, the ratio maintains the same regularity of $u$ and $v$ also across $R(u)$. Let us be more precise: let us consider a fixed $u\in H^1(B_1)$ solution to \eqref{equv} with $A\in C^{k,\alpha}(B_1)$ for some $k\in\mathbb N$, $\alpha\in(0,1)$ and $S(u)\cap B_1=\emptyset$. Then, let us consider solutions $v$ to \eqref{equv} in $B_1$ with $Z(u)\subseteq Z(v)$. We are going to prove Theorem \ref{schauderR(u)}.

\begin{proof}[Proof of Theorem \ref{schauderR(u)}]
The desired regularity for the ratio is obtained after localizing the problem on a point lying on $R(u)$ and then by a covering argument. Using the same notation of \S \ref{sec-curved}, $\Gamma=R(u)$ is locally described as the graph of a function $\varphi$, the upper side as its epigraph and the lower side as its hypograph. Then, the Schauder estimates for $w$ are obtained separately in $\Omega^+_\varphi\cap B_1$ and $\Omega^-_\varphi\cap B_1$ as in the higher order boundary Harnack principle of Subsection \ref{subs:HOBH}. Then, regularity across $\Gamma$ follows by Lemma \ref{gluinglemma2}. The only point to show is that actually $w$ is continuous across $\Gamma$. This is due to the following limit: taking $z_0\in\Gamma$ and small $t>0$
\begin{eqnarray*}
w(z_0+t\nu^+(z_0))&=&\frac{v(z_0+t\nu^+(z_0))}{u(z_0+t\nu^+(z_0))}\\
&=&\frac{v(z_0+t\nu^+(z_0))-v(z_0)}{t}\left(\frac{u(z_0+t\nu^+(z_0))-u(z_0)}{t}\right)^{-1}\to\frac{\partial_{\nu^+} v(z_0)}{\partial_{\nu^+} u(z_0)}
\end{eqnarray*}
as $t\to0^+$. Similarly
\begin{equation*}
w(z_0+t\nu^-(z_0))=\frac{v(z_0+t\nu^-(z_0))}{u(z_0+t\nu^-(z_0))}\to\frac{\partial_{\nu^-} v(z_0)}{\partial_{\nu^-} u(z_0)}.
\end{equation*}
Using the $C^1$ regularity of $u,v$ and the fact that the normal derivative of $u$ on $z_0\in R(u)$ is actually non zero, we are saying that the values of $w$ from above and below $R(u)$ agree
$$\lim_{t\to0^+}w(z_0+t\nu^+(z_0))=\lim_{t\to0^+}w(z_0+t\nu^-(z_0))=w(z_0).$$

The covering argument is standard once we fix $u$, and its nodal set $Z(u)$, and we observe that $\inf_{B_{1/2}\cap Z(u)}|\partial_{\nu}u|\geq L>0$. In fact, any $z\in B_{1/2}$ is either a point of $Z(u)$ or of its complement and, in both cases, it exists a small radius $r_z>0$ such that on the ball $B_{r_z}(z)$ the regularity estimate holds true. Hence, by compactness one can extract a finite covering of $B_{1/2}$ with the same property.
\end{proof}

\subsection{Gradient estimates across general zero sets in dimension $n=2$}\label{sub.n2}

Aim of this Subsection is to address, in the two-dimensional case, local gradient estimates for ratios $w=v/u$ of solutions $u,v$ to \eqref{equv} with $Z(u)\subseteq Z(v)$ and without restrictions on $Z(u)$. The main result of our study is Theorem \ref{teodim2}.

Since in the two-dimensional setting the singular set $S(u)$ is a locally finite set of isolated points (see \cite{GarLin, Han, Bers}), we can start our analysis by localizing the problem around a given singular point. First, we prove local gradient estimates for solutions to \eqref{eqnodal} which depends on the vanishing order of $u$ at the isolated singular point and then, by the gluing Lemma and a standard covering argument, we provide local $C^{1,1/N_0(r)-}$ regularity for the ratio $w=v/u$ which depends on the maximum attained by the vanishing order of $u$ in any smaller ball $\overline B_r\subset B_1$ (see \eqref{maxV}).\\
Naturally, the estimate obtained depends on the fixed nature of the nodal set $Z(u)$. Indeed, in the forthcoming work \cite{TerTorVit2}, we will show that the estimate is in fact uniform in a suitable compact class $\mathcal S_{N_0}$ of solutions with bounded frequency.

\subsubsection{Local gradient estimates on a nodal domain via conformal mapping}
Let us localize the problem around an isolated singular point.  As we have remarked in the Introduction, although the $C^{1,1-}$ regularity of the nodal curves away from singularities is a direct consequence of the implicit function theorem, proving that the same regularity holds true up to the singular points requires using the following lemma.
We would like to remark that the proof follows the ideas contained in \cite{HarWin1}  \gio{(see also  \cite{Cheng} for a similar result in the case of smooth coefficients).}

\begin{Lemma}\label{lemmacurves}
Let $n=2$ and let $u$ be a solution to \eqref{equv} in $B_1$ with $A\in C^{0,1}(B_1), A(0)=\mathbb I$. Assume that $S(u)\cap B_1=\{0\}$ and $N=\mathcal V(0,u)\in\mathbb N\setminus\{0,1\}$. Then, exactly $N$ nodal curves of class $C^{1,1-}$ meet at $0$ with equal angle $\pi/N$ and forming $2N$ nodal regions.
\end{Lemma}

\begin{proof}
The proof follows the strategy developed in \cite{HarWin1,HelHofTer}, with some adjustments made to address the presence of Lipschitz coefficients. The proof involves first applying a quasiconformal map and then deducing regularity estimates for the branch of the nodal set close to the singularity using a Cauchy integral formula around the singular point.\\

{\it Step 1: change of coordinates.}
First, defining
$$C=\frac{A}{\sqrt{\mathrm{det}A}},\qquad F=A\nabla\left(\frac{\sqrt{\mathrm{det}A}-1}{\sqrt{\mathrm{det}A}}\right),$$
then, $u$ solves
\begin{equation}
-\mathrm{div}(C\nabla u)=F\cdot\nabla u\qquad\mathrm{in \ }B_1,
\end{equation}
$C$ still uniformly elliptic, $C(0)=\mathbb I$, $\mathrm{det}C\equiv 1$, $C\in C^{0,1}(B_1)$, $F\in L^\infty(B_1)$. Let us define for later purposes the variable coefficients as $c_{11}=a, c_{12}=c_{21}=b, c_{22}=c$ with $ab-c^2\equiv 1$ (being $a,b,c$ Lipschitz continuous functions in the unit ball). Then, let us construct a solution $Y$ to
\begin{equation}\label{eqX}
\mathrm{div}(C\nabla Y)=0\qquad\mathrm{in \ }B_{r_0},
\end{equation}
with the property that $|\nabla Y|>0$ in a given small ball $B_{r_0}$. The existence of such a solution can be proved as follows. Defining $C_r(z)=C(rz)$ for $0<r<1$, let us consider $v_r$ the unique solution to
\begin{equation*}
\begin{cases}
\mathrm{div}(C_r\nabla v_r)=0 &\mathrm{in \ } B_1,\\
v_r=x_1 &\mathrm{on \ } \partial B_1.
\end{cases}
\end{equation*}
Since the sequence $v_r$ is uniformly bounded in $C^{1,\alpha}(\overline{B_1})$, for any given $0<\alpha<1$, it converges, up to a subsequence, to $v:=x_1$, which is the unique solution to
\begin{equation*}
\begin{cases}
\Delta v=0 &\mathrm{in \ } B_1,\\
v=x_1 &\mathrm{on \ } \partial B_1.
\end{cases}
\end{equation*}
Hence, fixed $0<\delta<1$, the uniform convergence $\nabla v_r\to\vec e_1$ in $\overline B_1$ ensures that $|\nabla v_r|>1-\delta$ in $B_1$ for any $0<r\leq r_0(\delta)$. Hence, defining $Y(z)=r_0v_{r_0}(z/r_0)$, it solves \eqref{eqX} with $|\nabla Y|>1-\delta$ in $B_{r_0}$ and $Y=x_1$ on $\partial B_{r_0}$. Then, let us define $X$ by
\begin{equation}\label{systY}
\begin{cases}
X_x=b Y_x+c Y_y\\
X_y=-(a Y_x+b Y_y).
\end{cases}
\end{equation}
Let us consider
\begin{align*}
J=\left(\begin{array}{cc}
0&1\\
-1&0
\end{array}\right), \qquad \mathrm{with}\quad J^{-1}=J^T=-J,
\end{align*}
satisfying in particular $Jz=z^\bot$ for any point $z$, where $z\cdot z^\bot=0$. Hence, one can see that \eqref{systY} is equivalent to $\nabla X=JC\nabla Y$. Moreover, since $CJC=J$, $X,Y$ are both solutions to \eqref{eqX} with $|\nabla X|,|\nabla Y|>0$ in $B_{r_0}$ and are $C$-orthogonal in the sense that
\begin{equation}
\nabla X=JC\nabla Y,\qquad \nabla Y=J^{-1}C\nabla X.
\end{equation}
Moreover, being solutions to \eqref{eqX}, they belong to $C^{1,1-}(B_{r_0})$. Let us define the following diffeomorphism
$$\Theta(x,y)=(X(x,y), Y(x,y)),$$
which is of class $C^{1,1-}$. The Jacobian matrix $J_\Theta$ belongs to $C^{0,1-}$ and its determinant is
$$\mathrm{det} J_\Theta=C\nabla X\cdot\nabla X=C\nabla Y\cdot\nabla Y>0.$$
Hence one can define also the inverse
$$\Theta^{-1}(X,Y)=(x(X,Y),y(X,Y)),$$
which is still of class $C^{1,1-}$. Let us now define $u=U\circ\Theta$; that is,
\begin{equation}
U(X,Y)=u(x,y).
\end{equation}
Then, $U$ solves
\begin{equation}\label{tildeFU}
-\Delta U=\tilde F\cdot\nabla U
\end{equation}
in a small ball, where the field $\tilde F$ in the drift is still bounded and depends on $\Theta^{-1}$ and $F$. From now on we will prove properties for $U$ being the same properties inherited directly by u through composition with $\Theta$ (we will discuss in details this fact in the last part of the proof).\\

\gio{
{\it Step 2: the Cauchy integral formula.}
The proof follows the paper by Hartman and Wintner \cite{HarWin} (see also \cite[Theorem 2.1]{HelHofTer}). Before stating the result we set few notations: let $z=(x,y)=x+iy = r e^{i\theta}$, where $x,y\in \R, r=|z|$ and $\theta=\mathrm{Arg}(z)\in[0,2\pi)$. Given $U$ as before, since $\mathcal{V}(0,U)= N\geq 2$, we have
$$
U(z)= O(|z|^N),\quad |\nabla U(z)|=O(|z|^{N-1})\qquad\mbox{as }|z|\to 0^+.
$$
Now if we set $w=i\overline{\nabla U}=U_y + i U_x$, since $|w(z)|=O(|z|^{N-1})$, then by \cite[Section 7]{HarWin} the following Cauchy formula holds true
\begin{equation}\label{e:cauchy}
2\pi i \frac{w(\zeta)}{\zeta^{N-1}}=\int_{\partial B_R}\frac{w(z)}{z^{N-1}(z-\zeta)}\, dz \ - \ \int_{B_R}\frac{\Delta U(z)}{z^{N-1}(z-\zeta)}\, dx \, dy,
\end{equation}
where $R>0$ is fixed and $\zeta$ belongs to a small neighbourhood of the origin. Notice that the first term in \eqref{e:cauchy} is smooth since the integral as no singularities on the circle, and the double integral over the disk is absolutely convergent by \eqref{tildeFU}. First, we show that the right hand side in \eqref{e:cauchy} is $\log$-Lipschitz continuous with respect to $\zeta$, indeed
\begin{eqnarray*}
\left|\int_{B_R}\frac{\Delta U(z)}{z^{N-1}}\left(\frac{1}{z-\zeta_1}-\frac{1}{z-\zeta_2}\right)\, dx \, dy\right| &\leq& \int_{B_R}\frac{|\nabla U(z)|\cdot|\tilde F(z)|}{|z|^{N-1}}\frac{|\zeta_1-\zeta_2|}{|z-\zeta_1|\cdot|z-\zeta_2|}\, dx \, dy\\
&\leq& C|\zeta_1-\zeta_2|\cdot|\log |\zeta_1-\zeta_2||.
\end{eqnarray*}
Therefore, there exists a complex-valued $\log$-Lipschitz function $\xi$ (and hence belonging to $C^{0,1-}$) such that
\begin{equation}\label{xi}
\nabla U(z) = U_x(z)+iU_y(z)=\overline z^{N-1}\xi(z)=r^{N-1}e^{-i(N-1)\theta}\xi(z).
\end{equation}
Finally, integrating the last equation one gets the existence of a real-valued $\log$-Lipschitz function $\xi_0$ such that $\xi_0(0)=0$ and

\begin{equation}\label{Uprimaexp}
U(z)= \frac{r^{N}}{N}\left(\mathrm{Re}(\xi(0))\cos(N\theta)+\mathrm{Im}(\xi(0))\sin(N\theta)+\xi_0(z)\right),
\end{equation}
where ultimately
$$
\xi(0)\neq 0\quad\mbox{and}\quad|\xi_0(z)|\leq C |z|^{N+\alpha},
$$
for every $\alpha \in (0,1)$.\\
}

{\it Step 3: parametrization of the curves.}
Since $\xi(0)\neq0$, we can rewrite \eqref{Uprimaexp} as
$$U(z)= \frac{r^{N}}{N}\left(|\xi(0)|\cos(N\theta-\theta_0)+\xi_0(z)\right),
$$
for some $\theta_0\in[0,2\pi)$. So, on any of the $2N$-branches of $Z(U)$, the function
\begin{equation}
\theta(z)=\frac{1}{N}\left(\theta_0\pm\arccos(\xi_0(z)/|\xi(0)|)+ k \pi\right),\qquad\mathrm{for \ some\ }k=0,1,\dots,N,
\end{equation}
is in fact $\log$-Lipschitz. Let us parametrize a general branch with the ODE
\begin{equation}
\dot z(t)=\frac{i\nabla u}{r^{N-1}}(z(t))=i e^{-i(N-1)\theta(z(t))}\xi(z(t)),
\end{equation}
for $t\in[0,1]$. First, let us remark that $|\dot z(t)|$ is bounded on any given branch. So the parametrization $t\mapsto z(t)$ is Lipschitz continuous. Hence, thanks to the latter information and the regularity of $\xi$ and $\theta$ in the variable $z$ we get by composition that $t\mapsto z(t)$ belongs to $C^{1,\omega}$ where $\omega(s)=s|\log s|$ for $s>0$ is the $\log$-Lipschitz modulus of continuity (which imply $C^{1,1-}$).

Ultimately, all of the information obtained on the nodal lines of $U$ translates into information on the nodal lines of $u$. Indeed, having $C(0)=\mathbb I$, the $C$-orthogonality of $X$ and $Y$ is a standard orthogonality at the origin (this preserves the geometry of the crossing nodal lines at $0$) and so the number of curves at singular points and their meeting angles remain the same. Moreover, since $\Theta$ and $\Theta^{-1}$ are $C^{1,1-}$, the regularity of the branches of $Z(u)$ is $C^{1,1-}$.
\end{proof}

Now, given $u$ solution to \eqref{equv} with $S(u)\cap B_1=\{0\}$ and $\mathcal V(0,u)=N\in\mathbb{N}\setminus\{0,1\}$, we perform a linear transformation which gives $A(0)=\mathbb I$ and hence turns any nodal domain of $u$ into a asymptotically conical domain with aperture $\pi/N$, that is
$$u^*(z)=u(A^{1/2}(0)z).$$
Then,
\begin{equation*}
\mathrm{div}\left(A^*\nabla u^*\right)=0\qquad\mathrm{in \ } B_1\qquad\mbox{and}\qquad S(u^*)\cap B_1=\{0\},
\end{equation*}
where
\begin{equation}\label{Astar}
A^*(z)=A^{-1/2}(0)A(A^{1/2}(0)z)A^{-1/2}(0)|\det A^{1/2}(0)|,\qquad \mathrm{with \ }A^*(0)=\mathbb I.
\end{equation}
Notice that every connected nodal component $\Omega_{u}$ turns into $\Omega_{u^*}=\Omega_{\pi/N}$ which stands for a nodal region of $u^*$ which is asymptotic to a cone $C_{\pi/N}$ of aperture $\pi/N$ centered at $0$ by Lemma \ref{lemmacurves} (see Figure \ref{figure1}).\vspace{-1cm}

\begin{figure}[h]
\begin{tikzpicture}[scale =1.3]
 \begin{scope}
\clip  plot [smooth,rotate=60] coordinates {(0,0)  (0.09,0.55) (0.4,1.17) (1,0.8) (1.4,1.2) (2.4,0.6) (2.13,0.54) (2.4,-0.03)};
\fill[gray!30] (1.5,0) arc (0:360:1.5cm and 1.5cm);
\end{scope}
\begin{scope}
\clip (0,0) -- (0,1.5)-- (1.5,1.5)--(1.5,0);
\fill[gray!30] (1.5,0) arc (0:180:1.5cm and 1.5cm);
\end{scope}
\begin{scope}
\clip plot [smooth] coordinates {(0,0)  (0.09,0.55) (0.4,1.01) (1,1.1) (1.4,1.6) (2.4,1) (2.13,0.54) (2.4,-0.03)};
\fill[white] (1.5,0) arc (0:360:1.5cm and 1.5cm);
\end{scope}
\begin{scope}
    \clip (1.7,0) arc (0:360:1.7 and 1.7);
\draw [black,thick] plot [smooth] coordinates {(0,0)  (0.09,0.55) (0.4,1.01) (1,1.1) (1.4,1.6) (2.4,1) (2.13,0.54) (2.4,-0.03)};
\draw [thick,->] (0.4,1.01)--(0.41,1.02) node[anchor=north west] {$\scriptstyle \gamma_2$};
\draw [black,thick] plot [smooth,rotate=+120] coordinates {(0,0)  (0.09,0.55) (0.4,1.01) (1,1.1) (1.4,1.6) (2.4,1) (2.13,0.54) (2.4,-0.03)};
\draw [black,thick] plot [smooth,rotate=-120] coordinates {(0,0)  (0.09,0.55) (0.4,1.01) (1,1.1) (1.4,1.6) (2.4,1) (2.13,0.54) (2.4,-0.03)};
\draw [black,thick] plot [smooth,rotate=-60] coordinates {(0,0)  (0.09,0.55) (0.4,1) (1,1.5) (1.4,1.2) (2.4,1) (2.13,0.54) (2.4,-0.03)};
\draw [black,thick] plot [smooth,rotate=60] coordinates {(0,0)  (0.09,0.55) (0.4,1.17) (1,0.8) (1.4,1.2) (2.4,0.6) (2.13,0.54) (2.4,-0.03)};
\draw [thick,->,rotate=60] (0.395,1.17)--(0.41,1.18) node[anchor=north east ] {$\scriptstyle \gamma_1$};
\draw [black,thick] plot [smooth,rotate=-180] coordinates {(0,0)  (0.07,0.45) (0.35,1.07) (0.8,0.7) (1.05,1.1) (2.25,0.5) (2.03,0.54) (2.4,-0.03)};
\end{scope}
      \draw[black,dashed] (-1.7,-0.9) -- (1.7,0.9);
      \draw[black,dashed] (-1.7,0.9) -- (1.7,-0.9);
      \draw[black,dashed] (0,-1.7) -- (0,1.7);
      \draw (1.5,0) arc (0:360:1.5cm and 1.5cm);
      \draw (0.1,-0.05) node[anchor=north] {$\scriptstyle 0$};
      \draw (-0.31,0.9) node {$\Omega_{\pi/N}$};
      \draw (-1.1,1.1) node[anchor=south] {$C_{\pi/N}$};
      \draw[<->] (-0.5,0.5*0.9/1.7) arc (120:90:1 and 1.6);
      \draw (-0.16,0.29) node {$\scriptscriptstyle{\frac{\pi}{N}}$};
      \draw (1.7,0.5) node {$B_1$};
      \draw (-1.8,0.09) node {$Z(u)$};
           \end{tikzpicture}
       \caption{This picture represents the general scenario near singular points in $\R^2$ (in this case, the asymptotic cone corresponds to $C_{\pi/N}$ with $N=3$).
}
\label{figure1}
       \end{figure}
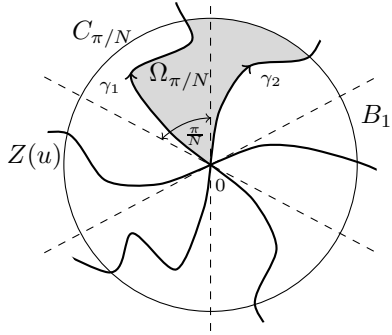

Assuming $A(0)=\mathbb I$, the idea now is to compose the solution $u$ to \eqref{equv} with a conformal mapping which opens a conical nodal region of angle $\gamma=\pi/N$ up to $\pi$; that is,
\begin{equation*}
\Psi(z)=z^N=(x+iy)^N=\zeta
\end{equation*}
with inverse given by $\Psi^{-1}(\zeta)=\zeta^{1/N}=z$, where the latter is well defined in the following way: if $\zeta=\rho e^{i\theta}$ where $\rho=|\zeta|>0$ and $\theta=\mathrm{Arg}(\zeta)\in[0,2\pi)$, then $\zeta^{1/N}=\rho^{1/N}e^{i\theta/N}$. Let us remark here that the complex power function $z^\alpha$ with $\alpha\in(0,1)$ is well defined up to removing a half line starting from $0$ and on compact subsets of this domain it is $\alpha$-H\"older continuous up to the origin; for instance locally in $\mathbb C\setminus\{(\rho,2\pi-\delta) \ : \ \rho>0\}$ for a chosen small $\delta>0$. After applying the conformal mapping, one can see that there exists a small radius $R>0$ such that $\Psi^{-1}(\Omega_{\pi}\cap B_R)\subseteq \Omega_{\pi/N}\cap B_1$. We are going to show now that $\Psi^{-1}(\partial\Omega_{\pi}\cap B_R)=\partial\Omega_{\pi/N}\cap B_{R^{1/N}}$. Let us consider two parametrizations $\gamma_1(t)$ and $\gamma_2(t)$ for $t\in[0,1]$ of the branches meeting at the vertex $0$ which form the boundary of $\Omega_{\pi/N}$. Then $\gamma_1,\gamma_2$ are of class $C^{1,1-}$ by Lemma \ref{lemmacurves} and we can assume that $\gamma_1(0)=\gamma_2(0)=0$ and the tangent vectors to the arcs at the vertex are $\gamma_1'(0)=\vec e_1=(1,0)$ and $\gamma_2'(0)=\vec v\neq\vec e_1$ with angle between $\vec v$ and $\vec e_1$ given by $\pi/N$ (see again Figure \ref{figure1}). Then, adapting the proof of \cite[Proposition 2.7]{HelHofTer}, the following result holds true.
\begin{Lemma}\label{openingboundary}
The boundary $\Psi(\partial\Omega_{\pi/N}\cap B_{R^{1/N}})$ can be parametrized by a curve of class $C^{1,1/N-}$ given by the juxtaposition of the arcs $\tilde\gamma_i(t)=(\gamma_i(t^{1/N}))^N$ for $i=1,2$ which are both of class $C^{1,1/N-}$, where the juxtaposition of $\gamma_i$ parametrizes the boundary $\partial\Omega_{\pi/N}\cap B_1$.
\end{Lemma}
\proof
By a computation on both the arcs $i=1,2$
\begin{equation*}
\tilde\gamma_i'(t)=\gamma_i'(t^{1/N})\left(\frac{\gamma_i(t^{1/N})}{t^{1/N}}\right)^{N-1}.
\end{equation*}
Since $\gamma_i$ belongs to $C^{1,1-}$ then $\gamma'_i$ belongs to $C^{0,1-}$. Nevertheless, having $\gamma_i(0)=0$ we have by Lemma \ref{lem1} that $\gamma_i(\tau)/\tau$ belongs to $C^{0,1-}$. Then, composing with $t^{1/N}$ we obtain that $\tilde\gamma'_i$ belongs to $C^{0,1/N-}$. Finally, we remark that the tangents to the new arcs at the vertex are given by
\begin{equation*}
\tilde\gamma'_i(0)=\lim_{t\to0^+}\left(\frac{\gamma_i(t^{1/N})}{t^{1/N}}\right)^N=\left(\gamma'_i(0)\right)^N,
\end{equation*}
and the angle between $\vec v^N$ and $\vec e_1^N$ is $\pi$. Hence, the curve given by
\begin{equation*}
\tilde\gamma(t)=\begin{cases}
\tilde\gamma_2(-t) &\mathrm{if \ } t\in[-1,0]\\
\tilde\gamma_1(t) &\mathrm{if \ } t\in[0,1]
\end{cases}
\end{equation*}
is a $C^{1,1/N-}$-parametrization of the new boundary.
\endproof

Finally, we can prove gradient estimate on a nodal domain of $u$ for solutions to \eqref{eqnodal}.

\begin{proof}[Proof of Theorem \ref{teoBnodal}]
Let $a\in \R$ be such that $|u|^a \in L^1(B_1)$ and $w$ be a solution to \eqref{eqnodal}. Then the first thing to do is to compose with the linear transformation which turns a nodal domain of $u$ into $\Omega_{\pi/N}$ and does not affect regularity of solutions. Let us assume without loss of generality that $u>0$ in the considered $\Omega_u$. Hence $w^*(z)=w(A^{1/2}(0)z)$ solves
\begin{equation}\label{eqnodal2}
\mathrm{div}\left((u^*)^a B^*\nabla w^*\right)=0\qquad\mathrm{in \ } \Omega_{\pi/N}\cap B_1,
\end{equation}
in the sense of a $H^1(\Omega_{\pi/N}\cap B_1,u^*(z)^adz)$-function such that
\begin{equation*}
\int_{\Omega_{\pi/N}\cap B_1}(u^*)^aB^*\nabla w^*\cdot\nabla\phi=0,
\end{equation*}
for any test function belonging to $C^{\infty}_c(B_1)$. The new matrix is given by
\begin{equation*}
B^*(z)=A^{-1/2}(0)B(A^{1/2}(0)z)A^{-1/2}(0)|\det A^{1/2}(0)|,\qquad \mathrm{with \ }B^*(0)=\mathbb I
\end{equation*}
since $A(0)=B(0)$. From now on, for the sake of simplicity, we indicate $A^*,B^*,u^*,w^*$ by $A,B,u,w$.

Thus, by considering just one nodal region, $u$ satisfies
\begin{equation*}
\begin{cases}
\mathrm{div}\left(A\nabla u\right)=0 &\mathrm{in \ }\Omega_{\pi/N}\cap B_1\\
u>0 &\mathrm{in \ }\Omega_{\pi/N}\cap B_1\\
u=0 &\mathrm{on \ }\partial\Omega_{\pi/N}\cap B_1.
\end{cases}
\end{equation*}
Now, let $\overline u, \overline w$ be respectively the compositions of $u$ and $w$ with $\Psi^{-1}$; that is,
\begin{equation}\label{compconf}
\overline u=u\circ\Psi^{-1},\qquad \overline w=w\circ\Psi^{-1}.
\end{equation}
Then, it is easy to see that actually $\overline u$ is solution to
\begin{equation*}
\mathrm{div}\left(\overline A\nabla \overline u\right)=0 \quad\mathrm{in \ }\Omega_{\pi}\cap B_R,\qquad
\overline u>0 \quad\mathrm{in \ }\Omega_{\pi}\cap B_R,\qquad
\overline u=0 \quad\mathrm{on \ }\partial\Omega_{\pi}\cap B_R,
\end{equation*}
where the new matrix is uniformly elliptic with same constants $\lambda,\Lambda$, $\overline A(0)=\mathbb I$ and it satisfies

\begin{equation}\label{matrixnew}
\overline A(\zeta)=[J_\Psi AJ_\Psi^T]\circ\Psi^{-1}(\zeta) \,|\mathrm{det} J_{\Psi^{-1}}(\zeta)|.
\end{equation}

Nevertheless, $\overline w$ solves

\begin{equation}\label{eqnodal2}
\mathrm{div}\left(\overline u^a \overline B\nabla \overline w\right)=0\qquad\mathrm{in \ } \Omega_{\pi}\cap B_R,
\end{equation}
in the sense of a $H^1(\Omega_{\pi}\cap B_R,\overline u(z)^adz)$-function such that
\begin{equation*}
\int_{\Omega_{\pi}\cap B_R}\overline u^a\overline B\nabla \overline w\cdot\nabla\phi=0,
\end{equation*}
for any test function belonging to $C^{\infty}_c(B_R)$. Here the new matrix is given by
\begin{equation*}
\overline B(\zeta)=[J_\Psi BJ_\Psi^T]\circ\Psi^{-1}(\zeta) \,|\mathrm{det} J_{\Psi^{-1}}(\zeta)|.
\end{equation*}

We stress that we have chosen a proper small radius $0<R<1$ such that $\Psi^{-1}(\Omega_{\pi}\cap B_R)\subseteq \Omega_{\pi/N}\cap B_1$ and so $\overline u$ and $\overline w$ solve the latter equations. We remark that regularity estimates away from $0$ follow by the previous analysis on the regular set. Hence, in order to ease the notation let us take $R=1$.

Denoting by $\Psi'(z)=Nz^{N-1}$ the complex derivative of $\Psi$, then
\begin{align*}
J_\Psi(z)=\left(\begin{array}{c|c}
\mathrm{Re}\Psi'(z)&-\mathrm{Im}\Psi'(z)\\\hline
\mathrm{Im}\Psi'(z)&\mathrm{Re}\Psi'(z)
\end{array}\right),
\end{align*}
and
\begin{equation*}
|\mathrm{det} J_{\Psi^{-1}}(\zeta)|=|(\Psi^{-1})'(\zeta)|^2=\frac{1}{N^2}|\zeta|^{2\frac{1-N}{N}}.
\end{equation*}

Hence, by the presence of the term $A\circ\Psi^{-1}$ in formula \eqref{matrixnew}, the coefficients of the new matrix $\overline A$ belong just to $C^{0,1/N}(\Omega_{\pi}\cap B_1)$ and no more. Let us remark here that the fact that $J_\Psi AJ_\Psi^T \,|\mathrm{det} J_{\Psi^{-1}}\circ\Psi|$ is Lipschitz continuous strongly relies on the necessary condition $A(0)=\mathbb I$ (without this condition $\overline A$ is not even continuous at $0$). Hence, regularity follows by composition since $\Psi^{-1}\in C^{0,1/N}$. Obviously, the same considerations above hold also for $\overline B$.

Nevertheless, the boundary $\partial\Omega_\pi$ is of class $C^{1,1/N-}$ by Lemma \ref{openingboundary}. Then, by classical regularity results for uniformly elliptic equations, $\overline u\in C^{1,1/N-}(\Omega_{\pi}\cap B_r)$ for any $r<1$. Then, one can compose our solutions with the standard straightening diffeomorphism $\Phi$ defined in \eqref{diffeo} which is of class $C^{1,1/N-}$; that is, $\tilde u=\overline u\circ\Phi$ and $\tilde w=\overline w\circ\Phi$. Then, up to possible dilations, $\tilde u$ solves
\begin{equation*}
\mathrm{div}\big(\tilde A\nabla \tilde u\big)=0 \quad\mathrm{in \ } B_1^+,\qquad
\tilde u>0 \quad\mathrm{in \ } B_1^+,\qquad
\tilde u=0 \quad\mathrm{on \ }B'_1,
\end{equation*}
where $\tilde A\in C^{0,1/N-}(B_1^+)$ and is defined as in \eqref{tildeA} by
\begin{equation*}
\tilde A=(J_\Phi^{-1}) (\overline A \circ\Phi) (J_\Phi^{-1})^T.
\end{equation*}
Even so, by the Hopf-Oleinik Lemma (we refer again to \cite{ApuNaz,KozKuz}), $\partial_y\tilde u>0$ on $\Sigma=\{y=0\}$ and $\tilde u\in C^{1,1/N-}(B_r^+)$ for any $r<1$. Then, proceeding as in the proof of the higher order boundary Harnack principle, $\tilde w$ is solution to
\begin{equation*}
\mathrm{div}\left(y^aC\nabla\tilde w\right)=0\qquad\mathrm{in \ } B_r^+
\end{equation*}
with
$$C=(\tilde u/y)^a\tilde B,\qquad\mathrm{and}\qquad \tilde B=(J_\Phi^{-1}) (\overline B \circ\Phi) (J_\Phi^{-1})^T,$$
which is of class $C^{0,1/N-}(B_r^+)$ by Lemma \ref{lem1} and still uniformly elliptic by the condition $\partial_y\tilde u>0$ on $\Sigma=\{y=0\}$ given by the Boundary Point principle. Hence, by Theorem \ref{EVENwithoutHA}, $\tilde w\in C^{1,1/N-}(B_{\overline r}^+)$ for any $\overline r<r$, and it satisfies
$$C\nabla\tilde w\cdot\vec e_n=0\qquad\mbox{on }\Sigma \cap B_r.$$
Up to composing back with the diffeomorphism and the conformal mapping, one obtains the desired regularity for $w$, having
\begin{equation*}
w=\overline w\circ\Psi=\tilde w\circ\Phi^{-1}\circ\Psi.
\end{equation*}
Moreover,
\begin{equation*}
B\nabla w\cdot\nu=0 \qquad \mathrm{on} \  \partial\Omega_{\pi/N}\cap B_1.
\end{equation*}
The latter boundary condition must be understood as $B\nabla w\cdot\nu_1=0$ on $\gamma_1$ and $B\nabla w\cdot\nu_2=0$ on $\gamma_2$ where $\nu_i$ is the normal vector to $\gamma_i$. Finally, the overdetermined condition at the vertex implies $B(0)\nabla w(0)=0$ which means $\nabla w(0)=0$ by uniform ellipticity of $B$.
\end{proof}

\subsubsection{Local gradient estimates for the ratio around isolated zeroes and covering}
Finally, we show that, as a consequence of Theorem \ref{teoBnodal}, we can address Theorem \ref{teodim2}.

\begin{proof}[Proof of Theorem \ref{teodim2}]
The proof can be divided into three steps.\\

{\it Step 1: gradient estimate on a nodal domain for the ratio.} First, let us localize the equation \eqref{equv} around an isolated singular point $\{0\}=S(u)\cap B_1$ with $\mathcal V(0,u)=N\in\mathbb N\setminus\{0,1\}$. Given $v$ another solution to \eqref{equv} with $Z(u)\subseteq Z(v)$, without loss of generality we may assume $A(0)=\mathbb I$. If not, as we have seen we can compose with the linear transformation related to the square root $A^{1/2}(0)$ which is symmetric and positive definite. Then, $u^*$ and $v^*$ are $L_{A^*}$-harmonic in $B_1$ with $A^*$ defined in \eqref{Astar}. This operation does not affect the regularity of our solutions and the property $Z(u^*)\subseteq Z(v^*)$ still holds true with $S(u^*)\cap B_1=\{0\}$. Hence, as we know by Proposition \ref{theratiosolution} the ratio $w=v/u$ solves
$$\mathrm{div}\left(u^2A\nabla w\right)=0\qquad\mathrm{in \ }B_1$$
and hence in particular on any nodal region $\Omega_{u}\cap B_1$. Applying Theorem \ref{teoBnodal} with $a=2$ we know that $w\in C^{1,1/N-}_{\loc}(\overline\Omega_{u}\cap B_1)$ and this hold true on any nodal component.\\

{\it Step 2: gluing the estimate across the nodal regions.} Thanks to the validity of a $C^{1,1/N-}$-estimate for the ratio $w$ on any nodal component of $u$ up to the singular vertex in the origin, we can apply the gluing Lemma \ref{gluinglemma2} getting eventually the desired local estimate in $B_1$ also across the curves. Moreover, the following boundary conditions are satisfied
$$
A\nabla w\cdot\nu=0 \quad\mathrm{on \ } \partial\Omega_{u}\cap B_1,\qquad
\nabla w(0)=0.$$
We remark that the estimate depends on $u$ and its nodal set $Z(u)$.\\

{\it Step 3: covering.} The estimate in the ball $B_{1/2}$ follows by a covering argument.
Every point $z$ in $B_{1/2}$ belongs either to $R(u)$, or $S(u)$ 
or $B_{1/2}\setminus Z(u)$. Hence, one can find a covering with small balls $B_{r_z}(z)$ where in each of them the desired estimate holds true. Hence, by compactness, one can extract a finite number of such balls for the covering.
Ultimately, the boundary condition
\begin{equation*}
A\nabla w\cdot\nu=0 \quad\mathrm{on \ } R(u)\cap B_{1},\qquad
\nabla w=0 \quad\mathrm{on \ } S(u)\cap B_{1},
\end{equation*}
follows by the uniform ellipticity of $A$ and the fact that $A(z_i)\nabla w(z_i)=0$ at any $z_i\in S(u)\cap B_{1}$.

\end{proof}

\section{A Liouville theorem for degenerate or singular problems on $\Sigma$}\label{Section4}
Aim of the Section is the proof of Theorem \ref{liouvilletheoremzero}. Through the Section we will always consider solutions $w \in H^1_\loc(\overline{\R^{n}_+},\rho dz)$ to \eqref{evenrho0} in the sense that
\begin{equation}\label{evenrho0.test}
\int_{\R^{n}_+} \rho \nabla w \cdot \nabla \phi =0,
\end{equation}
for every $\phi \in C^\infty_c(\R^{n})$. Moreover, we recall that the following hypothesis on $\rho \in L^1_\loc(\R)$ are assumed throughout the Section:
\begin{enumerate}
  \item $\rho(y)>0$ for every $y>0$;
  \item there exist $a>-1$ and $C>0$ such that
  $$
  \rho(y)\leq C(1+y^a),\quad\mbox{for every }y\in [0,+\infty).
  $$
\end{enumerate}
We start by recalling some general facts related to solutions to \eqref{evenrho0} in the following Lemmata.
\begin{Lemma}\label{caccio.sus}
 Let $w \in H^1_\loc (\overline{\R^n_+},\rho dz)$ be an entire solution to \eqref{evenrho0}. Then, there exists a universal constant $\tilde{C} >0$, such that
$$
\int_{B^+_r} \rho |\nabla w|^2 \leq \frac{\tilde{C} }{(R-r)^2}\int_{B^+_{R}\setminus B^+_r}\rho (w-\lambda)^2,
$$
for every $\lambda\in \R$, $0<r<R$.
\end{Lemma}
\begin{proof}
  The proof is classical but we sketch it for the sake of completeness. Let $\eta \in C^\infty_c(\R^{n})$ be a radially decreasing cut-off function such that $0\leq\eta\leq1$ and for some $0<r<R$,
  $$
 \mathrm{supp}\eta\subseteq B_R,\quad \eta\equiv 1 \mbox{ on }B_r, \quad\mbox{and}\quad|\nabla \eta|\leq \frac{2}{R-r}.
  $$
  Then, by testing \eqref{evenrho0.test} with $\phi = (w-\lambda)\eta^2$, we get
  $$
  \int_{B^+_R}\rho |\nabla w|^2 \eta^2 = -\int_{B^+_R} 2 (w-\lambda)\eta \nabla w\cdot \nabla \phi \leq \left(\int_{B^+_R}\rho  |\nabla w|^2 \eta^2\right)^{1/2}\left(\int_{B^+_R} \rho (w-\lambda)^2|\nabla \eta|^2\right)^{1/2}.
  $$
  Finally, by exploiting the definition of $\eta$, we get that
  $$
  \int_{B^+_r}\rho |\nabla w|^2 \eta^2\leq
    \int_{B^+_R}\rho |\nabla w|^2 \eta^2 \leq \frac{4}{(R-\rho)^2}\int_{B^+_R\setminus B^+_r} \rho (w-\lambda)^2.
  $$
\end{proof}
\begin{Corollary}\label{corollary2}
 Let $w \in H^1_\loc(\overline{\R^n_+},\rho dz)$ be an entire solution to \eqref{evenrho0}. Then
 \begin{enumerate}
   \item [\rm(a)] for every $i=1,\dots,n-1$, the weak derivative $\partial_{x_i} w \in H^1_\loc(\overline{\R^{n}_+},\rho dz)$ is a solution to \eqref{evenrho0};
   \item [\rm(b)] there exists $\tilde{C}>0$ such that if there exist $C>0$ and $\mu\in\R$ such that
   $$
   \int_{B_r^+} \rho w^2 \leq C r^\mu,\quad\mbox{for every }r\geq 1,
   $$
   we get for any $i=1,\dots,n-1$
   $$
   \int_{B_r^+} \rho (\partial_{x_i}  w)^2 \leq C \tilde{C} r^{\mu-2},\quad\mbox{for every }r\geq 1.
   $$
 \end{enumerate}
\end{Corollary}
\begin{proof}
  The part (b) of the result follows immediately from Lemma \ref{caccio.sus}. Indeed, by the Caccioppoli inequality we have
  $$
  \int_{B^+_r} \rho (\partial_{x_i} w)^2 \leq
\int_{B^+_r} \rho |\nabla w|^2 \leq \frac{\tilde{C} }{r^2}\int_{B^+_{2r}}\rho w^2 \leq C\tilde{C}r^{\mu-2},
$$
with $\tilde{C}>0$ as in Lemma \ref{caccio.sus}. Let us conclude the proof by showing that $\partial_{x_i} w \in H^1_\loc(\overline{\R^{n}_+},\rho dz)$ and satisfies \eqref{evenrho0.test}. Given $t\in (0,1/2)$, we set
$$
w^t_i(z)=\frac{w(z+te_i)-w(z)}{t} \in H^1_\loc(\overline{\R^{n}_+},\rho dz).
$$
Since the weight $\rho$ depends only on the variable $y$, we immediately deduce that $w^t_i$ is solution to \eqref{evenrho0} in the sense of \eqref{evenrho0.test}. First, by exploiting the relation between the first derivative and the finite difference quotient $w^t_i$, we have
\begin{align*}
\int_{B_r^+} \rho (w_i^t)^2 &\leq \int_{B_r^+}\left(\int_{0}^1 \rho (\partial_i w)^2(z+tse_i)\,ds\right)dz\\
&= \int_{0}^1 \left(\int_{B_r(-tse_i)^+} \rho (\partial_i w)^2(z)\,dz\right)ds \leq \int_{B_{r+1/2}^+} \rho (\partial_i w)^2
\end{align*}
for every $r\geq 1$ and $t \in (0,t_0)$. Therefore, since $w$ is a solution to \eqref{evenrho0}, by Lemma \ref{caccio.sus} we get
$$
\int_{B_r^+} \rho (w_i^t)^2 \leq \int_{B_{\frac32 r}^+} \rho (\partial_i w)^2  \leq \frac{\tilde{C}}{r^2} \int_{B_{2r}^+}\rho w^2
$$
for every $r\geq 1$, uniformly for $t\in (0,1/2)$. Ultimately, since $w^t_i$ is a solution to \eqref{evenrho0}, we infer that
$$
\int_{B_r^+}\rho |\nabla (w_i^t)|^2 \leq \frac{\tilde{C}}{r^2}\int_{B^+_{\frac32 r}} \rho (w_i^t)^2 \leq  \frac{\tilde{C}^2}{r^4} \int_{B_{2r}^+}\rho w^2 ,
$$
for every $r\geq 1$, which implies that $w_i^t$ is uniformly bounded in $H^1(B_r^+,\rho dz)$, for every $r\geq 1$. Therefore, by reflexivity and compact embedding, we get that, up to a subsequence, $w^t_i$ converges weakly in $H^1(B_r^+,\rho dz)$ and strongly in $L^2(B_r^+,\rho dz)$ to some $g \in H^1(B_r^+,\rho dz)$. Eventually, this function coincides with the weak derivative $\partial_{x_i}w$, namely for every $\varphi \in C^\infty_c(\R^{n})$ we have
$$
\int_{\R^n} w^t_i \varphi = \int_{\R^{n}} \varphi^{-t}_{i} w,\qquad\mbox{for every }t>0,
$$
which converges, up to a subsequence, to the definition of weak-derivative of $w$ along the direction $e_i$. Finally, by showing that $w_i^t$ converges to $\partial_{x_i} w$ strongly in $H^1_\loc(\overline{\R^n_+},\rho dz)$, we conclude that $\partial_{x_i} w$ is a solution of the desired equation. Indeed, by testing the equation satisfied by $w_i^t$ with $\phi (w_i^t - \partial_{x_i} w)$, with $\phi \in C^\infty_c(\R^n)$, we get
\begin{align*}
0 &= \int_{\R^n}\rho \nabla w_i^t \cdot \nabla (\phi(w_i^t-\partial_{x_i}w))\\
&= \int_{\R^n} \rho \left((w_i^t-\partial_{x_i} w) \nabla w_i^t \cdot \nabla \phi + \phi |\nabla w_i^t|^2 - \phi \nabla w_i^t \cdot \nabla \partial_{x_i} w\right)\\
&=o(1) + \int_{\R^n}\rho \left(|\nabla w_i^t|^2-|\nabla \partial_{x_i}w|^2\right)
\end{align*}
as $t \to 0^+$, where in the last equality we used the strong convergence in $L^2_\loc(\overline{\R^n_+},\rho dz)$ and the weak convergence in $H^1_\loc(\overline{\R^n_+},\rho dz)$. The convergence of norms implies the strong convergence in $H^1_\loc(\overline{\R^n_+},\rho dz)$.
\end{proof}
\begin{Lemma}\label{lemma3}
  Let $w \in H^1_\loc(\overline{\R^{n}_+},\rho dz)$ be a solution to \eqref{evenrho0} and suppose there exist $C,\mu>0$ such that
  \begin{equation}\label{grow}
  \int_{B_r} \rho w^2 \leq C r^\mu, \quad\mbox{for every }r\geq 1.
  \end{equation}
 Then $w$ is a polynomial in the variable $x=(x_1,\cdots,x_{n-1})\in \R^{n-1}$, with coefficients depending only on the variable $y$.
\end{Lemma}
\begin{proof}
The result follows by iterating Corollary \ref{corollary2}. Let $\partial^k_{x_1} w$ be the $k$-th derivatives of $w$ with respect to the variable $x_1$. Then, by Corollary \ref{corollary2} we get that $\partial^k_{x_1} w$ solves \eqref{evenrho0} and satisfies
$$
  \int_{B_r} \rho (\partial^k_{x_1} w)^2 \leq C(\tilde{C})^k r^{\mu-2k}, \quad\mbox{for every }r\geq 1.
  $$
  Thus, let $k\in \N$ be the first integer such that $\mu-2k<0$. Considering $r\to+\infty$ we get that $\partial^k_{x_1}w\equiv 0$ in $\overline{\R^{n}_+}$ and so $w$ is of the form
  $$
  w(z)=\sum_{i=0}^{k-1}a_i(x_2,\dots,x_{n-1},y) x_1^i;
  $$
  that is, a polynomial in the variable $x_1$ of degree less or equal than $k-1$ with coefficients depending only on the variables $(x_2,\dots,x_{n-1},y)$. Finally, by repeating the same argument along the directions $x_i$, with $i=2,\dots,n-1$, we get the result.
\end{proof}
Finally, we can prove the Liouville type result in Theorem \ref{liouvilletheoremzero}.
\begin{proof}[Proof of Theorem \ref{liouvilletheoremzero}]
First, let us remark that the only solution $w$ to \eqref{evenrho0} depending only on the variable $y$ is constant.

 Let now $r\geq 1$. Then, the growth condition in \eqref{liouv.condition} gives
  $$
  \int_{B^+_r} \rho w^2 \leq C r^{n+1}(1+r^a)(1+r)^{2\gamma}\leq C r^{n+1+2\gamma +a^+},
  $$
for every $r\geq 1$. Thus, by Lemma \ref{lemma3} we get that $w$ is a polynomial in the variable $x=(x_1,\cdots,x_{n-1})\in \R^{n-1}$, with coefficients depending only on the variable $y$ and ultimately, by \eqref{liouv.condition}, we deduce that it is of degree at most $\gamma$.\\
Then, if $\gamma\in [0,2)$ we get that $w$ is of degree at most $1$ with respect to the variables $x_1,\dots,x_{n-1}$, and so of the form
$$
w(z)=a_{n}(y) + \sum_{i=1}^{n-1} a_i(y) x_i.
$$
On one hand, since $\partial_{x_i}w(z)=a_i(y) $ is a solution to \eqref{evenrho0}, we get that $a_i(y)\equiv a_i \in \R$ is constant. Thus, by substituting the whole function in \eqref{evenrho0} we get
$$
 \mathrm{div}\left(\rho\nabla w\right)= \mathrm{div}\left(\rho\nabla a_{n}(y)\right),\qquad
 \lim_{y\to0^+}\rho \,\partial_yw=\lim_{y\to0^+}\rho \,\partial_y a_{n}(y)
$$
from which it follows that necessary $a_{n}(y)\equiv a_{n}$ is a constant too. Finally, the result follows by taking $t=a_{n}$ and $b=(a_1,\dots,a_{n-1})$. Ultimately, if $\gamma \in [0,1)$ we immediately deduce that $w(z)\equiv t$.
\end{proof}

\bigskip

{\bf Data availability statement:} this manuscript has no associated data.

\end{document}